\documentclass[10pt]{amsart}
\usepackage{color}
\usepackage{enumerate}
\usepackage{amssymb}
\usepackage{multirow}

\newtheorem{teo}{Theorem}[section]
\newtheorem{lemma}[teo]{Lemma}
\newtheorem{prop}[teo]{Proposition}
\newtheorem{cor}[teo]{Corollary}

\theoremstyle{definition}
\newtheorem{definiz}[teo]{Definition}
\newtheorem{example}[teo]{Example}
\newtheorem{remark}[teo]{Remark}
\newtheorem{notation}[teo]{Notation}
\newtheorem{assumption}[teo]{Assumption}

\newcommand{\R}{\mathbb R}

\newcommand{\C}{\mathbb C}

\newcommand{\So}{\mathcal{S}}

\newcommand{\B}{\mathbb B}
\newcommand{\SF}{\mathbb S}

\DeclareMathOperator{\pv}{\wedge \mspace{-9.5mu}_* \ }

\newcommand{\N}{\mathbb N}

\newcommand{\HH}{\mathbb H}

\theoremstyle{remark} 

\numberwithin{equation}{section}

\begin{document}

\title[one-slice preserving functions.]{S-regular functions which preserve a complex slice}
\author[A. Altavilla]{A. Altavilla${}^{\dagger,\ddagger}$}\address{Altavilla Amedeo: Dipartimento Di Matematica, Universit\`a di Roma ``Tor Vergata", Via Della Ricerca Scientifica 1, 00133, Roma, Italy} \email{altavilla@mat.uniroma2.it}

\author[C. de Fabritiis]{C. de Fabritiis${}^{\dagger}$}\address{Chiara de Fabritiis: Dipartimento di Ingegneria Industriale e Scienze
Matematiche, Universit\`a Politecnica delle Marche, Via Brecce Bianche, 60131,
Ancona, Italia} \email{fabritiis@dipmat.univpm.it}
\thanks{${}^{\dagger}$GNSAGA of INdAM, ${}^{\ddagger}$FIRB 2012 {\sl Geometria differenziale e teoria geometrica delle funzioni} and SIR grant {\sl ``NEWHOLITE - New methods in holomorphic iteration''} n. RBSI14CFME}

\date{\today }


\begin{abstract} 
We study global properties of quaternionic slice regular functions  (also called \textit{s-regular}) defined on symmetric slice domains.
In particular, thanks to new techniques and points of view, we can characterize the property
of being one-slice preserving in terms of the projectivization of the vectorial part of the function. We also define a ``Hermitian'' product on slice regular functions which gives us the possibility to express the $*$-product of two s-regular functions in terms of the scalar product of  suitable functions constructed starting from $f$ and $g$.  
Afterwards we are able to determine, under different assumptions,
when the sum, the $*$-product and the $*$-conjugation of two slice regular functions preserve a complex slice. 
We also study when the $*$-power of a slice regular function has this property or 
when it preserves all complex slices.
To obtain these results we prove two factorization theorems: in the first one,
we are able to split a slice regular function into the product of two functions: one keeping track of the zeroes and the other 
which is never-vanishing;
in the other one we give necessary and sufficient conditions for a slice regular function (which preserves all  complex slices)
to be the symmetrized of a suitable slice regular one.
\end{abstract}
\maketitle

\section{Introduction}

Since the seminal paper by Gentili and Struppa \cite{G-St}, several articles \cite{C-G-S-St,C-GC-S,Gal-S,G-Sal-S,G-S-St,G-V,G-M-P,G-P,R-W,S,V}
and some monographs \cite{C-S-St-1,C-S-St-2,G-S-St} have been published in the field of (quaternionic) slice regular functions. The theory was mainly built to allow quaternionic polynomials to be regular
and to mime, in some sense, the theory of complex holomorphic functions.
For this reason many works in this field address the search for analogies with the theory
of holomorphicity.

This paper, and the results enclosed therein, points out some global
behaviour of slice regular functions which are proper of the realm of quaternions
and have not a ``complex analogous''.  This fact is 
investigated by means of the new techniques
concerning the $*$-product of slice regular functions
partially introduced in \cite{A-dF} and developed in the present work.

To state some of these results we begin with some notation and known fact.
The main reference for this part is the monograph \cite{G-S-St}.

The space of quaternions $\HH$ is the four dimensional associative algebra
generated by $1,i,j,k$ with usual relations $i^{2}=j^{2}=k^{2}=-1$ and $ij=-ji=k$.
The division algebra of quaternions can be split as $\HH=\R\oplus {\rm Im}\HH$ and inside
${\rm Im} \HH\simeq\R^3$ we identify the sphere of imaginary units
$$
\SF:=\{q\in\HH\,|\,q^{2}=-1\}=\{q_{1} i+q_{2} j+q_{3} k\,|\, q_{1}^{2}+q_{2}^{2}+q_{3}^{2}=1\}\simeq S^{2}\subset {\rm Im}\HH.
$$

A quaternion $q\in\HH$ will then be written in the following ways:
$$
q=q_{0}+q_{1}i+q_{2}j+q_{3}k=q_{0}+I\beta=q_{0}+\vec q,
$$
where, $q_{i}\in\R$ for $i=0,1,2,3$, $\beta\in\R$, $I\in \SF$ and $\vec q\in{\rm Im}\HH$.
Each of the previous notations will be useful for some purpose. In particular
we point out that the product of two quaternions $q=q_{0}+\vec q, p=p_{0}+\vec p$ 
can be written as
\begin{equation*}
qp=q_{0}p_{0}-\langle \vec q,\vec p\rangle +q_{0}\vec p +p_{0}\vec q+\vec q\wedge \vec p,
\end{equation*}
where $\langle \vec q,\vec p\rangle$ and $\vec q\wedge \vec p$ denote the standard Euclidean  and vectorial product of ${\rm Im}\HH\simeq \R^{3}$, respectively.
In $\HH$ we will consider the usual conjugation $q=q_{0}+\vec q\mapsto q^{c}=q_{0}-\vec q$, so that $qq^{c}=q_{0}^{2}+\langle \vec q,\vec q\rangle=q_{0}^{2}+q_{1}^{2}+q_{2}^{2}+q_{3}^{2}=|q|^{2}$.
We also denote the open unitary ball as
$$
\mathbb{B}=\{q\in\HH\,|\,|q|< 1\}.
$$

We now pass to quaternionic functions theory.
The definition of slice regularity is based on the observation that is possible to 
\textit{unfold} the space of quaternions in the following way:
\begin{equation*}
\HH=\bigcup_{I\in\SF}\C_{I},\quad \C_{I}:={\rm Span}_{\R}(1,I).
\end{equation*}
We call \textit{slice} any complex line of the form $\C_J$, for $J\in\SF$.
From this point of view one can see that is possible to consider a non-constant complex
structure over $\HH\setminus\R$ that is \textit{tautological} with respect
to the slice (see for instance~\cite{AAtwistor,G-Sal-S}). Before recalling the notion of regularity we establish the family of domains
where we will define our spaces of functions.
\begin{assumption}
In the whole paper $\Omega\subset\HH$ will denote a symmetric slice domain, see~\cite{G-S-St}, that is a domain such that
\begin{itemize}
\item for any $q=\alpha+I\beta\in\Omega$, the set $\SF_{q}:=\{\alpha+J\beta\,|\,J\in\SF\}$ is contained in $\Omega$;
\item the intersection $\Omega \cap\R$ is non-empty.
\end{itemize}
\end{assumption}

Notice that  $\SF_{q}$ consists of the single point $q$ is $q\in\R$  and it a 2-dimensional sphere if $q\notin\R$.

\begin{definiz}
A function $f:\Omega\to\HH$ is said to be slice regular if all its restriction $f_{J}:=f|_{\C_{J}\cap\Omega}$ are holomorphic with respect to the tautological complex structure, i.e. for any $J\in\SF$, it holds
$$
\frac{1}{2}\left(\frac{\partial}{\partial\alpha}+J\frac{\partial}{\partial\beta}\right)f_{J}(\alpha+J\beta)\equiv 0.
$$
The family of slice regular functions over a fixed domain $\Omega$ is a real vector space, a right $\HH$-module and 
in this paper it will be denoted by $\So(\Omega)$ (in some of the references, see e.g.~\cite{G-M-P}, this symbol denotes the space of continuous slice functions which are not necessarily regular).
\end{definiz}

Examples of slice regular functions are quaternionic polynomials and quaternionic
power series (in their domain of convergence), with right coefficients.

Thanks to the \textit{Representation Formula} (see~\cite[Theorem 1.15]{G-S-St} and \cite{C-G-S-St}), it is known that the hypothesis on
the symmetry of the domain is not restrictive, while the one involving the role of the
real line is included to avoid some degenerate cases, see e.g. \cite{A-CVEE}.
Again, thanks to the Representation Formula, if $g:\Omega\cap\C_I\to\HH$ is any holomorphic function, then it is possible to extend
it in a unique way to a slice regular function $f:\Omega\to\HH$, such $f$ will be called \textit{regular extension}
of $g$, (see~\cite{G-S-St}, p. 9).

The following two natural subsets of $\So(\Omega)$ are of particular interest for our research.
\begin{definiz}
A slice regular function $f\in\So(\Omega)$ is said to be
\begin{itemize}
\item \textit{slice preserving} if $f(\Omega\cap\C_{I})\subset\C_{I}$, for all $I\in\SF$;

\item \textit{one-slice preserving} if there exists $J\in\SF$ such that $f(\Omega\cap\C_{J})\subset\C_{J}$;
for a fixed $J\in\SF$, these functions will also be called $\C_J$-\textit{preserving}.
\end{itemize}
The set of slice preserving functions will be denoted by $\So_{\R}(\Omega)$, while the set of 
$\C_{J}$-preserving functions by $\So_{J}(\Omega)$.
\end{definiz}

Examples of slice preserving functions are quaternionic polynomials and quaternionic
power series (in their domain of convergence) with \textit{real} coefficients; examples of $\C_J$-preserving functions are quaternionic polynomials and quaternionic
power series (in their domain of convergence) with right coefficients which belong to $\C_J$.

Slice preserving and one-slice preserving functions are special slice regular functions which
are more likely to be studied with classical complex methods. 
In particular, the last ones can give a deeper insight on the general case
being, in some sense, the middle point between the theory of holomorphic functions 
and the genuine quaternionic case (see for instance~\cite{V}). 

Thanks to the following result (see~\cite{C-GC-S,G-M-P}), it is possible to split a slice regular function
into a linear combination of any basis of $\HH$ with slice preserving regular functions as coefficients.

\begin{prop}\label{GMP} 
Let $\{1,I_0,J_0,K_0\}$ be a basis of $\HH$. Then the map 
$$
\left(\So_\R(\Omega)\right)^4\ni (f_0,f_1,f_2,f_3)\mapsto f_0+f_1I_0+f_2J_0+f_3K_0\in \So(\Omega)
$$
is bijective. In particular it follows that given any $f\in \So(\Omega)$ there exist and are unique $f_0,f_1,f_2,f_3\in\So_\R(\Omega)$ such that 
$$
f=f_0+f_1I_0+f_2J_0+f_3K_0.
$$
Moreover if $I_0\in\SF$ then $f\in \So_{I_0}(\Omega)$ iff $f_2\equiv f_3\equiv 0$. 
\end{prop}

The previous result allows us to define in a more natural way the regular conjugate and the $*$-product, already
presented in~\cite{G-S-St}.

\begin{definiz}
Given any basis $\{1,I_0,J_0,K_0\}$ of $\HH$ and  $f=f_0+f_1I_0+f_2J_0+f_3K_0\in\So(\Omega)$ we define its conjugate as $f^c=f_0+f_1I_0^c+f_2J_0^c+f_3K_0^c$.
\end{definiz}
In the previous definition, if $\{I_0,J_0,K_0\}$ is an orthonormal basis of ${\rm Im}\HH$, then $f^c=f_0-f_1I_0-f_2J_0-f_3K_0$.

From the conjugation just defined we can isolate the \textit{real} and \textit{vectorial} parts of a slice regular function.

\begin{definiz}
Given $f\in\So(\Omega)$, we define the slice-regular functions $f_0$ and $f_v$  on $\Omega$ by
$f_0=\frac{f+f^c}2$, $f_v=\frac{f-f^c}2$.
\end{definiz}

Clearly, fixed $f\in\So(\Omega)$, the functions $f_0$ and $f_v$ only depend on the conjugation, in particular $f=f_0+f_v$, $f^c=f_0-f_v$
and, according to the notation used in Proposition~\ref{GMP}, $f_v=f_1I_0+f_2J_0+f_3K_0$.

In this new language the $*$-product can be defined by associativity and distributivity over $\So_{\R}(\Omega)$ in the following way.
\begin{definiz}\label{starprod}
Given $f=f_0+f_1I_0+f_2J_0+f_3K_0,g=g_0+g_1I_0+g_2J_0+g_3K_0\in\So(\Omega)$ their $*$-product is given by
\begin{align*}
f*g&=
f_0g_0-f_1g_1-f_2g_2-f_3g_3+f_0(g_1I_0+g_2J_0+g_3K_0)+g_0(f_1I_0+f_2J_0+f_3K_0)\\
&+(f_2g_3-f_3g_2)I_0+(f_3g_1-f_1g_3)J_0+(f_1g_2-f_2g_1)K_0,
\end{align*}
where the products in the right hand side of the equality are the pointwise products (functions $f_0,\dots,f_3$ are slice preserving functions, so our definition coincides with the one given in~\cite{G-S-St}). 
\end{definiz}

\begin{remark}
Notice that  for slice preserving and $\C_J$-preserving functions, the $*$-product has special features. First of all if
$f\in\So_\R(\Omega)$ and $g\in\So(\Omega)$, then we have $f*g=g*f=fg$ (that is, the $*$-product $f*g$ coincides with the pointwise product $fg$). For any $J\in\SF$ and any couple of functions $f,g\in\So_J(\Omega)$ we have $f*g=g*f$.
Finally if $\rho_1,\rho_2\in \So_\R(\Omega)$ and $a_1,a_2\in\HH$ then $(\rho_1a_1)*(\rho_2a_2)=\rho_1\rho_2a_1a_2$.
\end{remark}

\begin{remark}
As $\So_{\R}(\Omega)$ is a unitary commutative ring, Proposition~\ref{GMP} can be interpreted as the fact 
that $\So(\Omega)$ is a free module of rank $4$ over $\So_{\R}(\Omega)$ and $\So_{I_{0}}(\Omega)$ is a 
free submodule of rank $2$, for any $I_{0}\in\SF$. 
Following this point of view, we provide characterizations of the desired functions in terms of cosets of suitable submodules: in some sense, results like Theorem~\ref{thm1}, \ref{primo-coniugio} and \ref{secondo-coniugio} can be seen as a parametric description, while Proposition~\ref{prodstarfI} displays ``bilinear'' equations.
\end{remark}

Given $f,g\in\So(\Omega)$, the formula in Definition~\ref{starprod} can be simplified using the operators
\begin{equation}\label{prodotto-vettore-scalare}
(f\pv  g)=\frac{(f*g)-(g*f)}2,
\qquad 
\langle f, g\rangle_*=(f*g^c)_0.
\end{equation}
In fact, in terms of the notation of Proposition~\ref{GMP} we can rewrite the above intrinsic expressions in the following form:
\begin{align*}
f\pv  g &=(f_2g_3-f_3g_2)I_0+(f_3g_1-f_1g_3)J_0+(f_1g_2-f_2g_1)K_0 \\
\langle f, g\rangle_*&=f_0g_0+f_1g_1+f_2g_2+f_3g_3.
\end{align*}
At last, with the use of the above operators, we can write
$$
f*g=f_0g_0-\langle f_v, g_v\rangle_*+f_0g_v+g_0f_v+f_v\pv  g_v,
$$
in complete analogy with the quaternionic product in $\HH$.

Using the conjugate function and the $*$-product, it is possible to define the symmetrized of a slice regular
function. In some sense, this function plays the same formal role of the square norm in the space of quaternions.
\begin{definiz}
Given $f=f_0+f_1i+f_2j+f_3k\in\So(\Omega)$, we define its \textit{symmetrized} function $f^s$
as 
$$
f^s=f*f^c=\langle f, f\rangle_*=f_0^2+f_1^2+f_2^2+f_3^2.
$$ 
\end{definiz}
We remark that for any $f\in\So(\Omega)$ its symmetrized function $f^s$, which is called normal function and is denoted by $\mathcal {N}(f)$ by some other authors (see for instance~\cite{G-P}) because it gives the norm of $f$ in the $*$-algebra of slice regular functions, belongs to $\So_\R(\Omega)$.

We now spend some words on the zero locus of a slice regular function (see~\cite{G-S, G-S-St,G-P}).
We start with a notation: given any $q=\alpha+I\beta\in\HH\setminus\R$ we set
$$
\SF_{\alpha+I\beta}:=\{\alpha+J\beta\,|\,J\in\SF\}.
$$
It is known that, if $f\in\So(\Omega)\setminus\{0\}$, then its zero locus is closed with empty interior; moreover it consists of isolated points
and isolated $2$-spheres of the form $\SF_{\alpha+I\beta}$.
In the particular cases of slice preserving or one-slice preserving functions, the previous assertion specializes as follows.
If $f$ belongs to $\So_\R(\Omega)\setminus\{0\}$, then its zero set consists of isolated \textit{real} points
and isolated $2$-spheres of the form $\SF_{\alpha+I\beta}$. If $f$ belongs to $\So_J(\Omega)\setminus\{0\}$, for some $J\in\SF$, then its zero set consists of isolated points belonging to $\C_J$
and isolated $2$-spheres of the form $\SF_{\alpha+I\beta}$.

As in the complex case, it is possible, for slice regular functions, to factor out a zero as follows.
Let $f$ be an element of $\So(\Omega)\setminus\{0\}$ and $\SF_{\alpha+I\beta}\subset \Omega$. Then
there exist $m,n\in\N$ and $p_1,\dots,p_n\in\SF_{\alpha+I\beta}$, with $p_\nu\neq p_{\nu+1}^c$ for all $\nu=1,\dots, n-1$, such that
\begin{equation}\label{factorzero}
f(q)=[(q-\alpha)^2+\beta^2]^m(q-p_1)*\dots*(q-p_n)*g(q),
\end{equation}
for some $g\in\So(\Omega)$ which does not have zeros in $\SF_{\alpha+I\beta}$.
Thanks to this factorization it is possible to introduce a notion of multiplicity of a zero in the following sense (see~\cite{G-S-St}).

\begin{definiz}
Let $f\in\So(\Omega)\setminus\{0\}$ and let $\SF_{\alpha+I\beta}\subset \Omega$ with $\beta\neq 0$.
Let $m,n\in\N$ and $p_1,\dots,p_n\in\SF_{\alpha+I\beta}$, with $p_\nu\neq p_{\nu+1}^c$ for all $\nu=1,\dots, n-1$, such that
Equation~\eqref{factorzero} holds for $f$ and some regular function $g$ which never vanishes in $\SF_{\alpha+I\beta}$.
We then say that $2m$ is the \textit{spherical multiplicity} of $\SF_{\alpha+I\beta}$ and that $n$ is the \textit{isolated multiplicity} of 
$p_1$.
If $x\in\R$ is such that $f(x)=0$, then we call \textit{isolated multiplicity} of $f$ at $x$ the number $k\in\N$ such that
$$
f(q)=(q-x)^kh(q),
$$
for some $h\in\So(\Omega)$ such that $h(x)\neq 0$.
\end{definiz}


\begin{remark}
It is known that if $\alpha+I\beta$ is such that $f(\alpha+I\beta)=0$, then any point in the set $\SF_{\alpha+I\beta}$ is a zero for $f^s$.
In particular $\So(\Omega)$ is an integral domain and $f_v^s\equiv0$ if and only if $f_v\equiv 0$. Indeed, since $\Omega$ is a slice domain if $f\not\equiv0$ and $g\not\equiv0$, it is enough to choose a real point $x_0\in\Omega$ at which neither $f$ nor $g$ vanish and $(f*g)(x_0)=f(x_0)g\left((f(x_0))^{-1}x_0f(x_0)\right)=f(x_0)g(x_0)\neq0$, where the first equality is due to Theorem~3.4 in~\cite{G-S-St}.
\end{remark}

We now have all the prerequisites to state the results contained in the paper. 
We list them by giving the essential structure of the paper.

Next section is devoted to give an intrinsic characterization of the family of one-slice preserving functions
in terms of the projectivization of their \textit{vectorial} part.

In Section 3 we prove two factorization results which will be used in Theorem~\ref{coniugato1}.
The first one (Proposition~\ref{nostroweierstrass}) is a Weierstrass-like factorization theorem for slice regular functions
with no non-real isolated zeroes.
In the second one (Theorem~\ref{radicesimmetrizzata}), we give necessary and sufficient condition for
a slice preserving function $\mu$ to be the symmetrized of the function $h\in\So(\Omega)$, that is $\mu=h^{s}=h*h^{c}$.

In Section 4 we are able to provide necessary and sufficient conditions for two functions $f\in\So_{I_0}(\Omega)$ and $g\in\So_{J_0}(\Omega)$,
in order to determine whether their sum or $*$-product is $\C_{K_0}$-preserving for some $K_0\in\SF$. Then, in Theorem~\ref{prodstarfI}
we deepen our study of the $*$-product of $f$ and $g$ dropping out the hypothesis on $f$ and $g$ to be one-slice preserving.

Afterwards, in Section 5 we study the conjugation of two slice regular functions $f$ and $g$:
$$
h*f*h^c=g.
$$
We firstly impose the condition on $f$ (Theorem~\ref{primo-coniugio}) then on $h$ (Theorem~\ref{secondo-coniugio}) to be one-slice preserving in order to obtain necessary and sufficient conditions which guarantee $h*f*h^c$ is one-slice preserving.
Then in Theorem~\ref{coniugato1} we study again the conjugation prescribing $g\in\So(\Omega)$, $f\in\So_{I_0}(\Omega)$
and asking $h$ to be one-slice preserving. In Proposition~\ref{coniugato2} we study
the same problem, exchanging the requests on $f$ and $h$.

Exploiting the new results obtained in Section 5, in Section 6 we come back to 
$*$-products and we give necessary and sufficient conditions on $f, h\in\So(\Omega)$
in order that both $f*h$ and $h*f$ are one-slice preserving. 


In  the last section we examine the case, of $*$-power of a slice regular function $f$. After ruling out the trivial 
cases we give necessary and sufficient conditions for the $*$-power of the function $f$ to be either one-slice
preserving or slice preserving, finding an interesting link with some non-trivial result in real algebraic geometry.
In fact the function $f^{*d}$ belongs to $\So_\R(\Omega)$ if the two functions $f_0$ and $f_v$ are the zeros
of a particular binary form of degree $[(d-1)/2]$ with real zeroes.

All the listed results are enriched with explicit examples and remarks on the hypotheses.

At last we point out that, since the initial result (Proposition~\ref{GMP}), holds for the theory of slice regularity over 
a generic alternative $*$-algebra (see~\cite[Lemma 2.4]{G-P-series}),  
the new techniques we are going to introduce may be generalized to this wider context.

We end this introduction with two acknowledgements.
We warmly thank Prof. G. Ottaviani (Universit\`a di Firenze) for indicating us the results contained in~\cite{M} and Prof. L. Demeio (Universit\`a Politecnica delle Marche) for helping us with the explicit computation of the roots of $Q_d$ appearing
in Example~\ref{exbinforms}. 


\section{Preliminary results}

In this section we introduce a ``Hermitian'' product defined on $\So(\Omega)$ which allows us to read
the $*$-product in terms of the scalar product introduced in~\eqref{prodotto-vettore-scalare}. 
We also expose an intrinsic characterization of the family of one-slice preserving functions
based on Proposition~\ref{GMP}. 
\begin{definiz}\label{definizione-hermitiana}
The ``Hermitian'' product $\mathcal{H}_{*}:\So(\Omega)\times\So(\Omega)\to\So(\Omega)$ is given by
$$\mathcal{H}_{*}(f,g)=f*g^c,$$ for any $f,g\in\So(\Omega)$.
\end{definiz}

A trivial computation shows that the map $\So(\Omega)\ni f \mapsto \mathcal{H}_{*}(f,g)$ is left-$\So(\Omega)$-linear for any fixed $g\in\So(\Omega)$ and that $\left(\mathcal{H}_{*}(g,f)\right)^c=\mathcal{H}_{*}(f,g)$  for any $f,g\in\So(\Omega)$, ensuring that $\mathcal{H}_{*}$ is in some sense Hermitian.

For any orthonormal basis $i,j,k$ of  $\rm{Im} \HH$, Proposition~\ref{GMP} extends the natural relation between scalar and Hermitian product on $\HH$ to $\So(\Omega)$, giving an analogous to the formula in the complex case.

\begin{prop}
For any $f,g\in\So(\Omega)$ and any orthonormal basis $i,j,k$ of  $\rm{Im} \HH$, we have
\begin{equation}\label{relazione-hermitiana}
\mathcal{H}_{*}(f,g)=\langle f,g \rangle_*+\langle f,i*g \rangle_* i+\langle f, j*g\rangle_* j+\langle f,k*g \rangle_* k.
\end{equation}
\end{prop}
\begin{proof}
According to Proposition~\ref{GMP},  
let us write $f=f_0+f_1i+f_2j+f_3k$ and  $g=g_0+g_1i+g_2j+g_3k$.
By direct computation, the left hand term of~\eqref{relazione-hermitiana} amounts to
\begin{align*}
f_0g_0+f_1g_1+f_2g_2+f_3g_3+&(-f_0g_1+f_1g_0-f_2g_3+f_3g_2)i\\+&(-f_0g_2+f_1g_3+f_2g_0-f_3g_1)j+(-f_0g_3-f_1g_2+f_2g_1+f_3g_0)k.
\end{align*}
A straightforward application of the definition of $\langle \cdot,\cdot \rangle_*$ gives 
\begin{align*}
f_0g_0+f_1g_1+f_2g_2+f_3g_3&=\langle f,g \rangle_*\\
-f_0g_1+f_1g_0-f_2g_3+f_3g_2&=\langle f,i*g \rangle_*\\
-f_0g_2+f_1g_3+f_2g_0-f_3g_1&=\langle f, j*g\rangle_* \\
-f_0g_3-f_1g_2+f_2g_1+f_3g_0&=\langle f,k*g \rangle_*
\end{align*}
gives the conclusion.
\end{proof}

Now we turn to the issue of giving an intrinsic description of one-slice preserving functions; the quest for this result is originated from the need to characterize a function in this class without explicitly indicating the slice it preserves.

\begin{teo}\label{one-slice}
Let $f\in\So(\Omega)\setminus\So_{\R}(\Omega)$. The following are equivalent:
\begin{enumerate}[(i)]
\item $f$ is one-slice preserving;
\item $f_{v}^{s}$ has a square root $\sqrt{f_{v}^{s}}\in\So_{\R}(\Omega)$ and the map $\frac{f_{v}}{\sqrt{f_{v}^{s}}}$ is constant
outside the zero set of $f_{v}^{s}$;
\item the map $\Omega\cap\R\ni x\mapsto[f_{v}(x)]\in\mathbb{P}({\rm Im}\HH)$ is constant
outside the zero set of $f_{v}$.
\end{enumerate}
\end{teo}

\begin{proof}
\textit{(i)$\Rightarrow$(ii)} 
Denote by $\C_{I_{0}}$ the slice preserved by $f$. By Proposition~\ref{GMP}, the function  $f$ can be written as $f=f_{0}+f_{1}I_{0}$
and hence $f_{v}=f_{1}I_{0}$. This immediately entails $f_{v}^{s}=f_{1}^{2}$, so $f_{v}^{s}$ has a square root in $\So_{\R}(\Omega)$  and $\frac{f_{v}}{\sqrt{f_{v}^{s}}}\equiv I_{0}$ or $\frac{f_{v}}{\sqrt{f_{v}^{s}}}\equiv -I_{0}$ outside the zero set of $f_{v}^{s}$ according to the fact that the choosen square root is equal to $f_1$ or to $-f_1$.

\textit{(ii)$\Rightarrow$(iii)} 
It is enough to observe that $[f_{v}(x)]=\left[\left(\frac{f_{v}}{\sqrt{f_{v}^{s}}}\right)(x)\right]\in\mathbb{P}({\rm Im}\HH)$ for any $x\in\Omega\cap\R$ outside
the zero set of $f_{v}$.

\textit{(iii)$\Rightarrow$(i)} Since $f\notin\So_{\R}(\Omega)$, the function $f_{v}$ is not identically zero.
Now choose $x_{0}\in\Omega\cap\R$ such that $f_{v}(x_{0})\neq 0$ and $B_{0}$ a ball of center $x_{0}$
contained in $\Omega$ on which $f_{v}$ is never-vanishing;
thus the restriction of the function $f_{v}^{s}$ to $B_{0}$ has no zeroes. By Corollary 3.2 in~\cite{A-dF} there exists a square root
$\sqrt{f_{v}^{s}}\in\So_{\R}(B_{0})$. Since the map $\Omega\cap\R\ni x\mapsto[f_{v}(x)]\in\mathbb{P}({\rm Im}\HH)$ is constant
on $B_{0}\cap\R$, then also  the map $\Omega\cap\R\ni x\mapsto\left[\frac{f_{v}}{\sqrt{f_{v}^{s}}}(x)\right]$ is constant on
$B_{0}\cap\R$. Thus, since $B_{0}\cap\R$ is connected and $\SF\to \mathbb{P}({\rm Im}\HH)$ is a double-covering, there exists $I_{0}\in\SF$ such that $f_{v}(x)=\sqrt{f_{v}^{s}}(x)I_{0}$ 
for any $x\in B_{0}\cap\R$. Thanks to the Identity Principle this equality holds on $B_{0}$.
Now choose a basis $\{I_{0},J_{0},K_{0}\}$ of ${\rm Im}\HH$ and write $f_{v}=f_{1}I_{0}+f_{2}J_{0}+f_{3}K_{0}$,
with $f_{1},f_{2},f_{3}\in\So_{\R}(\Omega)$.
The uniqueness given by Proposition~\ref{GMP} shows that $f_{2}\equiv f_{3}\equiv 0$ on $B_{0}$ 
and a further application of the Identity Principle gives at last $f\in\So_{I_{0}}(\Omega)$.
\end{proof}

Applying Corollary 3.2 in~\cite{A-dF} and the above theorem to the case when  $\Omega_{I_{0}}=\Omega\cap\C_{I_{0}}$ is simply connected, i.e.: $\pi_{1}(\Omega_{I_{0}})=0$ for some, and then any, $I_{0}\in\SF$, gives
the following

\begin{cor}
Let $f\in\So(\Omega)\setminus\So_{\R}(\Omega)$ with $\Omega_{I_{0}}=\Omega\cap\C_{I_{0}}$ simply connected.
The following are equivalent:
\begin{enumerate}[(i)]
\item $f$ is one-slice preserving;
\item the zero set of $f_{v}$ does not contain non real isolated zeroes of odd multiplicity and $\frac{f_{v}}{\sqrt{f_{v}^{s}}}$ is constant
outside the zero set of $f_{v}^{s}$.
\end{enumerate}
\end{cor}

\section{Factorization theorems}

In this section we present a factorization theorem ``\`a la Weierstrass'' for slice regular functions without non-real isolated zeroes which generalizes the result obtained by Gentili and Vignozzi in~\cite{G-V}. This result allows us to give necessary and sufficient conditions, in terms of values taken on the real line, on a    
given function $\mu\in\So_{\R}(\Omega)\setminus\{0\}$ in order that there exists a one-slice preserving function $h$ such that $h^{s}=\mu$. Since we will need to define the analogous of the Weierstrass primary factors given on the unit disc in $\C$, for the present section the domain 
$\Omega\subset\HH$ is such that $\Omega_{I_{0}}=\Omega\cap\C_{I_{0}}$ is simply connected, i.e.: $\pi_{1}(\Omega_{I_{0}})=0$ for some, and then any, $I_{0}\in\SF$.

For each $m\in\N$ we introduce an analogous of the Weierstrass primary factor
$E_{m}:\HH\rightarrow\HH$ given by
$$
E_{m}(q)=(1-q)\exp\left(q+\frac{q^{2}}{2}+\dots+\frac{q^{m}}{m}\right).
$$
For any $f\in\So(\Omega)$ we consider the regular function defined on $\Omega$
$$
(E_{m}\circledast f)(q) =(1-f(q))*\exp_{*}\left(f(q)+\frac{f(q)^{*2}}{2}+\dots+\frac{(f(q))^{*m}}{m}\right).
$$
%
%
%
where $\exp_{*}$ is the quaternionic $*$-exponential introduced in~\cite{C-S-St-2} and studied in~\cite{A-dF}. 
The previous \textit{composition} operator denoted by $\circledast$, is the one 
introduced in~\cite[Definition 4.1]{R-W} (since $E_{m}$ has real coefficients the two definitions appearing in the
cited paper coincide). In particular if $f\in\So_{I_{0}}(\Omega)$, the factor $(E_{m}\circledast f)$ coincides with the regular extension of the
function $z\mapsto E_{m} (f(z))$ defined on $\C_{I_{0}}$ and it belongs to $\So_{I_{0}}(\Omega)$, too.
Notice that the $E_{m}$'s are slice preserving regular functions and, in analogy with what happens in the complex case (see~\cite{R}, Chapter XV), we have that, for any $m$ and for any $|q|\leq 1$
\begin{equation}\label{ineqE}
|1-E_{m}(q)|\leq |q|^{m+1}.
\end{equation}
Consider now a one-slice preserving function $f\in\So_{J}(\Omega)$, such that $f(\Omega)\subset\overline{\mathbb{B}}$, then
\begin{equation}\label{ineqEm}
|1-(E_{m}\circledast f)(q)|\leq \max_{\SF_{q}}|f(q)|^{m+1},
\end{equation}
for any $q\in\Omega$.
Indeed, since $1-E_{m}\circledast f$ belongs to $\So_{J}(\Omega)$, thanks to Proposition 2.6 in~\cite{deF-G-S} for any $q=\alpha+I\beta\in\Omega$ we have that
\begin{align*}
|1-(E_{m}\circledast f)(q)| \leq&  \max\{|1-(E_{m}(f(\alpha+J\beta))|,|1-(E_{m}(f(\alpha-J\beta))|\}\\
\leq &\max\{|f(\alpha+J\beta)|^{m+1},|f(\alpha-J\beta)|^{m+1}\}=\max_{\SF_{q}}|f(q)|^{m+1}
\end{align*}
In particular, if $f$ is not a constant of modulus $1$, then $\max_{\SF_{q}}|f(q)|$ is strictly less than $1$.


\begin{prop}\label{nostroweierstrass}
Given $f\in\So(\Omega)$ with no non-real isolated zeroes, suppose that $m\in\N$ is the multiplicity of $f$ at $0$,
then there exist $\mathcal{R},\mathcal{S}\in\So_{\R}(\Omega)$, $h\in\So(\Omega)$ with $h$ never vanishing, such that
\begin{equation*}
f(q)=q^{m}\mathcal{R}(q)\mathcal{S}(q)h(q),
\end{equation*}
where 
\begin{itemize}
\item $\mathcal{R}$ vanishes exactly at the real non-vanishing zeroes of $f$;
\item $\mathcal{S}$ vanishes exactly at the spherical zeroes of $f$.
\end{itemize}
\end{prop}

%

\begin{proof}
If $f$ has only a finite number of zeroes (both real and spherical), then the thesis is a direct consequence of a finite number of repeated applications of Theorem 3.36 in~\cite{G-S-St}.

Now we perform the proof in the case when both
real and spherical zeroes are infinite. The case in which one of these sets is finite is left to the reader.

We denote by  $\{b_{n}\}_{n\in\N}\subset\R\setminus\{0\}$ the sequence of the real non-vanishing zeros of $f$
and by $\{S_{n}\}_{n\in\N}$ the sequence of the spherical zeros of $f$, where all the zeros  are listed according to their multiplicities.

If $\Omega=\HH$, the statement is a particular case of the Weierstrass factorization theorem given in~\cite{G-V}.
That is, there exists a never vanishing function $h\in\So(\HH)$, and for all $n\in\N$, there exist $c_{n}\in S_{n}$, such that
\begin{equation*}
f(q)=q^{m}\mathcal{R}(q)\mathcal{S}(q)h(q),
\end{equation*}
where
\begin{equation*}
\mathcal{R}(q)=\prod E_{n}\left(\frac{q}{b_{n}}\right),\quad
\mathcal{S}(q)=\prod \left(E_{n}\circledast(qc_{n}^{-1})\right)^{s}.
\end{equation*}
%
%
%
%
If $\Omega\neq\HH$, Corollary 3.7 in~\cite{Gal-S} allows us to restrict to the case in
which $\Omega=\B $. 

For any $q_{0}\in\B\setminus \{0\}$, we set 
$$
M_{q_{0}}(q):=\left(q_{0}-\frac{q_{0}}{|q_{0}|}\right)*\left(q-\frac{q_{0}}{|q_{0}|}\right)^{-*}=\left(q-\frac{q_{0}}{|q_{0}|}\right)^{-*}*\left(q_{0}-\frac{q_{0}}{|q_{0}|}\right),
$$
which is the regular M\"obius transformation defined by Stoppato in~\cite{S}. 
Now we choose $c_{n}\in S_{n}$; thanks to Theorem 3.12 in~\cite{G-S-St}, we have that any closed ball centered at the origin with radius strictly 
less than $1$ contains only a finite number of real and spherical zeroes, so that 
$\lim|b_{n}|=1$ and $\lim|c_{n}|=1$.
Thus $\left|\frac{b_{n}}{|b_{n}|}-b_{n}\right|=|b_{n}|\left|\frac{1}{|b_{n}|}-1\right| \rightarrow 0$ and $\left|\frac{c_{n}}{|c_{n}|}-c_{n}\right|=|c_{n}|\left|\frac{1}{|c_{n}|}-1\right|\rightarrow 0$ for $n\to \infty$.

Now consider the following factors
\begin{equation*}
\mathcal{R}(q)=\prod \left(E_{n}\circledast M_{b_{n}}\right)(q),\quad
\mathcal{S}(q)=\prod \left(E_{n}\circledast M_{c_{n}}\right)^{s}(q)
\end{equation*}

According to~\eqref{ineqEm} and the estimates contained in Theorem 15.11 in~\cite{R}, each factor is well defined on $\B$ and belongs to $\So_{\R}(\B)$, moreover the product $q^{m}\mathcal{R}(q)\mathcal{S}(q)$ has the same zeroes of $f$.
At this point, arguing as in Theorems 4.31 and 4.32 in~\cite{G-S-St}, we can find 
a never-vanishing function $h\in\So(\B)$ such that
$$
f(q)=q^{m}\mathcal{R}(q)\mathcal{S}(q)h(q)
$$
and this concludes the proof.
\end{proof}

We now give necessary and sufficient conditions for a slice preserving function to be the symmetrized
of a one-slice preserving function.

\begin{teo}\label{radicesimmetrizzata}
Given $\mu\in\So_{\R}(\Omega)\setminus\{0\}$, there exists $h\in\So_{I_{0}}(\Omega)$ such that $h^{s}=\mu$ if and only if $\mu\geq 0$ on $\Omega\cap\R$ and the order of the real zeros of $\mu$ is even.
\end{teo}

Notice that the statement is independent from $I_{0}\in\SF$. Indeed, if
there exist $I_{0}\in\SF$ and $h=h_{0}+h_{1}I_{0}\in\So_{I_{0}}(\Omega)$ such that $h^{s}=\mu$, then for any $J_{0}\in\SF$ the function $\tilde{h}=h_{0}+h_{1}J_{0}\in\So_{J_{0}}(\Omega)$ satisfies $\tilde{h}^{s}=\mu$.

\begin{proof}
As usual we write $h=h_{0}+h_{1}I_{0}$ with $I_{0}\in\SF$ and $h_{0},h_{1}\in\So_{\R}(\Omega)$; then the equality $h^{s}=\mu$ becomes $\mu=h_{0}^{2}+h_{1}^{2}$.

The condition $\mu\geq 0$ on $\Omega\cap\R$ is trivially necessary.
If $\mu(x_{0})=0$ with $x_{0}\in\Omega\cap\R$ then $h(x_{0})=0$,
 so the slice preserving function $(q-x_{0})$ divides $h$ and hence $(q-x_{0})^{2}$ divides $h^{s}=\mu$ and
 the necessity of the second condition is also proved.

In order to prove the sufficiency of the above stated conditions,
we denote by $2m$ the multiplicity of $q=0$ as a zero of $\mu$,
by $\{b_{n}\}$ the real non-vanishing zeroes repeated accordingly to half their multiplicity,
by $\{S_{n}\}$ the sequence of spherical zeroes  repeated accordingly to half their multiplicity and by $c_{n}$ the element of $S_{n}\cap\C_{I_{0}}^{+}$. 

Thanks to Proposition~\ref{nostroweierstrass} it is possible to factorize $\mu$ 
as follows
$$\mu(q)=q^{2m}\mathcal{R}^{2}(q)\mathcal{S}(q)\nu(q)$$ 
where 
\begin{itemize}
\item $\mathcal{R}^{2}$ vanishes exactly at the real non-vanishing zeroes of $f$,
\item $\mathcal{S}$ vanishes exactly at the spherical zeroes of $f$,
\end{itemize}
both with the appropriate multiplicities and $\nu\in\So_{\R}(\Omega)$ never vanishing.

If $\Omega=\HH$ we have that
\begin{equation*}
\mathcal{R}(q)=\prod E_{n}\left(\frac{q}{b_{n}}\right),\quad
\mathcal{S}(q)=\prod \left(E_{n}\circledast(qc_{n}^{-1})\right)^{s}.
\end{equation*}
Since we chose $c_{n}$ all lying in the same $\C_{I_{0}}$, then we can write
$$
\mathcal{S}(q)=\prod_{*} \left(E_{n}\circledast(qc_{n}^{-1})\right)*\prod_{*} \left(E_{n}\circledast(qc_{n}^{-1})\right)^{c}.
$$

Thanks to Proposition 3.1 in~\cite{A-dF} there exists a square root $\sigma\in\So_{\R}(\HH)$ of $\nu$ and hence the
function 
$$
h(q)=q^{m}\mathcal{R}(q)\sigma(q)\tilde{\mathcal{S}}(q)
$$
where,
$$
\tilde{\mathcal{S}}(q)=\prod_{*} \left(E_{n}\circledast(qc_{n}^{-1})\right)
$$
belongs to $\So_{I_{0}}(\HH)$ and is such that $h^{s}=\mu$.

If $\Omega\neq\HH$, again Corollary 3.7 in~\cite{Gal-S} allows us to restrict to the case in
which $\Omega=\B $. In this case we have that
\begin{equation*}
\mathcal{R}(q)=\prod \left(E_{n}\circledast M_{b_{n}}\right)(q),\quad
\mathcal{S}(q)=\prod \left(E_{n}\circledast M_{c_{n}}\right)^{s}(q)
\end{equation*}
Again, since we chose $c_{n}$ all lying in the same $\C_{I_{0}}$, we then have
$$
\mathcal{S}(q)=\prod_{*} \left(E_{n}\circledast M_{c_{n}}\right)*\prod_{*} \left(E_{n}\circledast M_{c_{n}}\right)^{c}
$$
and the existence of a square root $\sigma\in\So_{\R}(\B)$ of $\nu$ 
allows us to conclude with the same argument as above.
%
\end{proof}

\begin{remark}
Notice that, given $\mu\in\So_{\R}(\Omega)\setminus\{0\}$, if there exists $\hat h\in\So(\Omega)$ such that $\hat h^{s}=\mu$, then trivially
$\mu\geq 0$ on $\Omega\cap\R$ and the order of the real zeros of $\mu$ is even.
Thus the previous result shows that the following conditions are equivalent:
\begin{itemize}
\item there exists $\hat h\in\So_{I_{0}}(\Omega)$ such that $\hat h^{s}=\mu$,
\item there exists $k\in\So(\Omega)$ such that $k^{s}=\mu$.
\end{itemize}
\end{remark}

\begin{example}
Consider $\mu:\HH\to\HH$ given by $\mu(q)=q^{2}+1$. On the real line $\mu$ 
is always strictly positive and it can be written as $\mu=h^{s}$, where
$h(q)=q+I_{0}$, for any $I_{0}\in\SF$. Nonetheless we can also write
$\mu=\hat h^{s}$, where $\hat h(q)=\cos(q)+\sin(q)i+q\cos(q)j+q\sin(q)k$ and
thanks to Theorem~\ref{one-slice}
 it is not difficult to show 
that $\hat h$ preserves no slice.
\end{example}

\section{Sum and $*$-product}

Let $f,h:\Omega\rightarrow \mathbb{H}$ be two slice regular functions such that 
$f$ is $\mathbb{C}_{I_{0}}$-preserving and $h$ is $\mathbb{C}_{J_{0}}$-preserving for some $I_0,J_0\in\SF$. We want to understand
when their sum and $*$-product is a $\mathbb{C}_{K_{0}}$-preserving regular function,
for a suitable $K_{0}\in\mathbb{S}$. If $I_{0}=\pm J_{0}$ then the question is trivial,
so in this section we suppose that $I_{0}\neq\pm J_{0}$; for the same reason we assume that $f$ and $h$ are not slice-preserving functions.

\begin{prop}
Let $f,h:\Omega\rightarrow \mathbb{H}$ be two slice regular functions such that $f=f_{0}+f_{1}I_{0}$ is $\mathbb{C}_{I_{0}}$-preserving and $h=h_{0}+h_{1}J_{0}$ is $\mathbb{C}_{J_{0}}$-preserving with 
$f_{0},f_{1},h_{0},h_{1}$ slice preserving functions. Then there exists $K_0\in\SF$ such that $f+h$ is $\mathbb{C}_{K_{0}}$-preserving if and only if  there exist $a , b \in\R\setminus\{0\}$ such that $K_0=a  I_0+ b  J_0$ and 
$ b  f_1-a  h_1\equiv0$.
\end{prop}

\begin{proof}
The sufficiency of the condition is trivial. In order to prove its necessity, notice that,
as $I_0$ and $J_0$ are linearly independent, they can be completed to a basis $I_0,J_0,L_0$ of ${\rm Im}\HH$ and we can write $K_0$ as $a  I_0+ b  J_0+\varepsilon L_0$ for suitable $a ,  b , \varepsilon\in\R$.
Then $f+h$ is equal to $f_0+h_0+f_1I_0+h_1J_0$. Now $f+h$ is $\C_{K_0}$-preserving if and only if there exist two slice preserving functions  $m_0,m_1$ such that 
$$f_0+h_0+f_1I_0+h_1J_0=f+h=m_0+m_1K_0=m_0+a  m_1I_0+ b  m_1 J_0+\varepsilon m_1 L_0.$$  
The bijectivity guaranteed by Proposition~\ref{GMP} entails that 
$f_1=a  m_1$, $h_1= b  m_1$ and $\varepsilon m_1=0$.
Since neither $f$ nor $h$ are slice preserving, then 
the function $m_1$ cannot be identically zero and $a $ and $ b $ are both different from zero. This implies that $\varepsilon=0$, $K_0=a  I_0+ b  J_0$ and
$ b  f_1-a  h_1\equiv 0$. 
\end{proof}


We start the discussion  on the $*$-product of two functions with a preliminary remark that sets the question
in the case their $*$-product belongs to $\So_{\R}(\Omega)$.

\begin{remark}
Let $f,h\in\So(\Omega)\setminus\{0\}$. Then $f*h$ belongs to $\So_{\R}(\Omega)$ if
and only if also $h*f$ belongs to $\So_{\R}(\Omega)$. In fact  if $f*h\in\So_{\R}(\Omega)$, then
$f^{s}h^{s}=(f^{c}*f)*(h*h^{c})=f^{c}*(f*h)*h^{c}=(f*h)*f^{c}*h^{c}=(f*h)*(h*f)^{c}$.
As both $f^{s}h^{s}$ and $f*h$ belong to $\So_{\R}(\Omega)$, then $(h*f)^{c}$
also lies in $\So_{\R}(\Omega)$ and therefore $h*f\in\So_{\R}(\Omega)$.

Moreover $f*h$ belongs to $\So_{\R}(\Omega)$ if
and only if $f,h^{c}$ are linearly dependent over $\So_{\R}(\Omega)$. In fact if there exist
$\alpha,\beta\in\So_{\R}(\Omega)\setminus\{0\}$, such that $\alpha f+\beta h^{c}\equiv 0$,
then $\alpha f*h+\beta h^{s}\equiv 0$. This implies that $\alpha f*h$ belongs to $\So_{\R}(\Omega)$
and therefore $f*h\in\So_{\R}(\Omega)$. Vice versa if $f*h\in\So_{\R}(\Omega)$, then $(f*h)*h^{c}=f*(h*h^{c})=h^{s}f$. As both $f*h$ and $h^{s}$ are not identically zero  because $\So(\Omega)$ is an integral domain, we are done.
\end{remark}

Now we turn to the non-trivial case. We first characterize, giving an explicit parametric description, the sets of regular functions
which preserve two different slices whose $*$-product also preserves a slice.

\begin{teo}\label{thm1}
Let $f,h:\Omega\rightarrow \mathbb{H}$ be two slice regular functions such that $f=f_{0}+f_{1}I_{0}$ is $\mathbb{C}_{I_{0}}$-preserving and $h=h_{0}+h_{1}J_{0}$ is $\mathbb{C}_{J_{0}}$-preserving with 
$f_{0},f_{1},h_{0},h_{1}$ slice preserving functions and $I_{0}$, $J_{0}$ linearly independent. Then there exists $K_0\in\SF$ such that $f*h$ is $\mathbb{C}_{K_{0}}$-preserving if and only if  there exist $a , b \in\R$, $\varepsilon\in\R\setminus\{0\}$ such that $K_0=a  I_0+ b  J_0+\varepsilon I_0\wedge J_0$,  
$f=f_1\left(\frac  b \varepsilon +I_0\right)$ and $h=h_1\left(\frac a \varepsilon +J_0\right)$.
\end{teo}
\begin{proof}
As above, the sufficiency of the condition is obtained by direct computation.
In order to prove its necessity, first of all we compute 
$$f*h=(f_0+f_1I_0)*(h_0+h_1J_0)=f_0h_0+h_0f_1I_0+f_0h_1J_0+f_1h_1I_0J_0.$$
 As 
$I_0J_0=-\langle I_0,J_0\rangle+I_0\wedge J_0$ we have that 
$$
f*h=f_0h_0-f_1h_1\langle I_0,J_0\rangle+h_0f_1I_0+f_0h_1J_0+f_1h_1I_0\wedge J_0.
$$
Then $f*h$ is $\C_{K_0}$-preserving for some $K_0\in\SF$ if and only if there exist two slice preserving functions  $m_0,m_1$ and $a , b , \varepsilon\in\R$ such that 
$$f_0h_0-f_1h_1\langle I_0,J_0\rangle+h_0f_1I_0+f_0h_1J_0+f_1h_1I_0\wedge J_0=m_0+m_1K_0=m_0+a  m_1I_0+ b  m_1 J_0+\varepsilon m_1 I_0\wedge J_0;$$  
this because $I_0,J_0,I_0\wedge J_0$ is a basis of  ${\rm Im}\HH$.
Again, the bijectivity guaranteed by Proposition~\ref{GMP} entails that 
$$
\begin{cases}
f_0h_0-f_1h_1\langle I_0,J_0\rangle=m_0,\\ f_1h_0=a  m_1,\\ h_1f_0= b  m_1, \\f_1h_1= \varepsilon m_1.
\end{cases}
$$
Since neither $f$ nor $h$ are slice preserving, the functions $f_1$ and $h_1$ are not identically zero. As $\So(\Omega)$ is an integral domain, then 
the function $m_1$ cannot be identically zero and $\varepsilon$ has to be different from zero. 
Thus  $m_1=\frac{f_1h_1}\varepsilon$ and therefore 
$f_0h_1= b  \frac{f_1h_1}\varepsilon$ and $h_0f_1=a \frac{f_1h_1}\varepsilon$.
These equalities can be written as $h_1\left(f_0-\frac b \varepsilon f_1\right)\equiv0$ and $f_1\left(h_0-\frac a \varepsilon h_1\right)\equiv0$. Again, since $f_{1}$ and $h_{1}$ are not identically zero, we obtain
$f_0=\frac b \varepsilon f_1$ and $h_0=\frac a \varepsilon h_1$ that is 
$f=f_1\left(\frac b \varepsilon+I_0\right)$ and $h=h_1\left(\frac a \varepsilon+J_0\right)$.
\end{proof}

\begin{remark}
We underline that, chosen a basis $I_0,J_0,K_0\in \SF$, the above result locates two real directions in the planes $\C_{I_0}$ and  $\C_{J_0}$, respectively generated by $\frac b \varepsilon+I_0$ and by $\frac a \varepsilon+J_0$, which ``give the angles'' of the rotations needed to obtain $f$ and $h$ from the slice preserving functions $f_1$ and $h_1$. That is, for any $p_0\in\HH$ real multiple of $\frac b \varepsilon+I_0$  and $q_0\in\HH$ real multiple of $\frac a \varepsilon+J_0$ and for any $f_1,h_1\in\So_\R(\Omega)$ the $*$-product of the functions $f=f_1p_0\in \So_{I_0}(\Omega)$ and $h=h_1q_0\in \So_{J_0}(\Omega)$ belongs to $\So_{K_0}(\Omega)$ and vice versa. 
\end{remark}

We now pass to another result related to the $*$-product of two regular functions.
In this case, given functions $f$ and $g$ we write explicit ``bilinear'' 
equations which characterize the fact that the $*$-product $f*g$ preserves a given slice.

\begin{prop}\label{prodstarfI}
Given $I_{0}\in\SF$ and  $f,g\in\So(\Omega)$, the following are equivalent
\begin{enumerate}[(i)]
\item  the $*$-product $f*g$ belongs to $\So_{I_{0}}(\Omega)$; 
\item  $\langle f, M_{0}*g^c\rangle_{*}\equiv 0$, for all $M_{0}\in\SF$ orthogonal to $I_0$;
\item  $\langle f^{c}, g*M_{0}\rangle_{*}\equiv 0$, for all $M_{0}\in\SF$ orthogonal to $I_0$.
\end{enumerate}
\end{prop}

\begin{proof}
$\textit{(i)}\Leftrightarrow \textit{(ii)}$
Choose an orthonormal basis $\{I_0,J_0,K_0\}$ of $\rm{Im}\HH$ and 
notice that, thanks to Definition~\ref{definizione-hermitiana}, condition $\textit{(i)}$ is equivalent to 
$\mathcal H_*(f,g^c)\in \So_{I_{0}}(\Omega)$. 
Now Equality~\eqref{relazione-hermitiana} ensures that 
$\mathcal H_*(f,g^c)\in \So_{I_{0}}(\Omega)$ if and only if 
$\langle f, J_0*g^c\rangle_*\equiv 0$ and 
$\langle f,K_0*g^c \rangle_* \equiv 0$ which, by linearity on $\R$, holds if and only if 
$\langle f, M_{0}*g^c\rangle_{*}\equiv 0$ for all $M_{0}\in\SF$ orthogonal to $I_0$.

$\textit{(ii)}\Leftrightarrow \textit{(iii)}$
Since $\langle f, M_{0}*g^c\rangle_{*}=\langle f^c, (M_{0}*g^c)^c\rangle_{*}=\langle f^c, g* (-M_{0})\rangle_{*}
=-\langle f^c, g * M_{0}\rangle_{*}$ the equivalence of the two conditions is immediately proven.
\end{proof}

\section{Conjugates}

We first establish a convention for following reference. This will simplify the presentation
of the forecoming results.

\begin{notation}\label{notation}
If $I_{0}, M_{0}\in\SF$ are linearly independent, throughout the rest of the paper 
we will keep the following notation.
We denote by $I_0,J_0,K_0$ the orthonormal basis of ${\rm Im}\HH$ such that $K_0$ is a positive multiple of $I_0\wedge M_0$ and $J_0=K_0I_0$. 
This gives that, up to the substitution of $I_{0},J_{0},K_{0}$ with $I_{0},-J_{0},-K_{0}$, we have  $M_0=a  I_0+ b  J_0$ for some $ b >0$ with $a ^2+ b ^2=1$.

If $I_{0}$ and $M_{0}$ are not orthogonal and we are interested only in the slices $\C_{I_{0}}, \C_{M_{0}}$, up to substituting $I_{0},J_{0},K_{0}$ with $-I_{0},J_{0},-K_{0}$
we can also suppose that $a>0$.
\end{notation}

In this section we study the behaviour of conjugates $h*f*h^{c}$ of slice regular maps, in the cases when either
the conjugator $h$ or the conjugated $f$ is one-slice preserving. In order to obtain more information on this
sort of functions, first of all we compute $h*f*h^c$ by means of the decomposition in real and vectorial part as introduced in
Section 1. Setting $f=f_0+f_v$ and $h=h_0+h_v$ we have $h^c=h_0-h_v$ and we can state the following lemma.
\begin{lemma}
Given $f,h\in\So(\Omega)$ it holds
\begin{equation}\label{coniugio}
h*f*h^c=[h*f*h^c]_0+\langle h_v,f_v \rangle_* h_v+h_0^2f_v+2h_0 h_v\pv f_v-(h_v\pv f_v)\pv h_v.
\end{equation}
\end{lemma}

\begin{proof}
The thesis is obtained thanks to the following chain of equalities 
\begin{equation*}
\begin{split}
h*f*h^c&=(h_0+h_v)*(f_0+f_v)*(h_0-h_v)=(h_0f_0-\langle h_v,f_v \rangle_*+h_0f_v+f_0h_v+h_v\pv f_v)*(h_0-h_v)\\
&=[h*f*h^c]_0-h_0f_0h_v+\langle h_v,f_v \rangle_* h_v+h_0^2f_v+h_0f_0h_v+h_0 h_v\pv f_v -h_0 f_v\pv h_v -(h_v\pv f_v)\pv h_v\\
&=[h*f*h^c]_0+\langle h_v,f_v \rangle_* h_v+h_0^2f_v+2h_0 h_v\pv f_v-(h_v\pv f_v)\pv h_v.
\end{split}
\end{equation*}
\end{proof}

Using the previous lemma, the first result we can prove is a complete classification
of the regular functions which satisfy the equality $h*f*h^{c}=h^{c}*f*h$.

\begin{prop}
Let $f, g\in\So(\Omega)$, then  $h*f*h^{c}=h^{c}*f*h$ if and only if either
$h_{0}\equiv 0$ or $f_{v}$ and $h_{v}$ are linearly dependent over $\So_{\R}(\Omega)$.
\end{prop}

\begin{proof}
Since $h^{c}=h_{0}-h_{v}$, from Equation~\eqref{coniugio} we have
\begin{equation*}
h^{c}*f*h=[h^{c}*f*h]_0+\langle h_v,f_v \rangle_* h_v+h_0^2f_v-2h_0 h_v\pv f_v-(h_v\pv f_v)\pv h_v.
\end{equation*}
A straightforward computation shows that $[h*f*h^{c}]_0=h_{0}^{2}f_{0}+f_{0}\langle h_{v},h_{v}\rangle_{*}=[h^{c}*f*h]_0$, hence
$h*f*h^{c}=h^{c}*f*h$ is equivalent to
\begin{equation}\label{doppioconiugio}
4h_{0}(h_v\pv f_v)\equiv 0.
\end{equation}
As $\So(\Omega)$ is an integral domain, \eqref{doppioconiugio} is equivalent to either $h_{0}\equiv 0$
or $h_v\pv f_v\equiv 0$. Thanks to Proposition 2.8 in~\cite{A-dF} the statement follows.
\end{proof}

We carry on our investigation on the behaviour of the conjugate by showing under
which (non-trivial) conditions on $h$ the function $h*f*h^c$ is one-slice preserving when $f$ is.

\begin{teo}\label{primo-coniugio}
Let $f\in\So_{I_0}(\Omega)\setminus\So_\R(\Omega)$ and $h\in \So(\Omega)\setminus\So_{I_0}(\Omega)$. Then there exists $M_0\in\SF$ such that 
$h*f*h^c\in\So_{M_0}(\Omega)$ if and only if  
\begin{enumerate}[(i)]
\item in the case $\C_{I_{0}}\neq \C_{M_{0}}$, by means of Notation~\ref{notation}, we can find $g\in\So_{I_0}(\Omega)\setminus\{0\}$ such that  
$$h=\left(1-\frac{a \pm1}{ b }K_0\right)*g;$$
%
%
%
\item in the case $\C_{I_{0}}= \C_{M_{0}}$, there exist $J_0\in\SF$ with $I_0\perp J_0$ and  $g\in \So_{I_0}(\Omega)\setminus\{0\}$ such that 
$h=J_0*g$. 
\end{enumerate} 
\end{teo}
\begin{proof}
We first consider the case in which $I_0$ and $M_0$ are linearly independent and adopt  Notation~\ref{notation}, according to the statement.

With respect to the chosen basis we can decompose $f_v$ and $h_v$ as $f_1I_0$ and $h_1I_0+h_2J_0+h_3K_0$, respectively. Since $h\notin \So_{I_0}(\Omega)$ then $h_2$ and $h_3$ are not both identically zero. 
As 
$$h_v\pv f_v=f_1h_3J_0-f_1h_2K_0\quad {\rm and}\quad(h_v\pv f_v)\pv h_v=f_1[(h_2^2+h_3^2)I_0-h_1h_2J_0-h_1h_3K_0]$$
from~\eqref{coniugio}, we have 
$$
h*f*h^c=[h*f*h^c]_0+f_1[(h_0^2+h_1^2-h_2^2-h_3^2)I_0+2(h_1h_2+h_0h_3)J_0+2(h_1h_3-h_0h_2)K_0].
$$
This function belongs to $\So_{M_0}(\Omega)$ if and only if there exists $m_1\in\So_\R(\Omega)$ such that
$$
\begin{cases}
h_0^2+h_1^2-h_2^2-h_3^2=a  m_1,\\
2(h_1h_2+h_0h_3)= b  m_1,\\
2(h_1h_3-h_0h_2)=0.
\end{cases}
$$
As $ b \neq 0$ we obtain 
\begin{equation}\label{doppia}
h_0^2+h_1^2-h_2^2-h_3^2=2\frac a  b (h_1h_2+h_0h_3) \qquad \mbox{and} \qquad h_1h_3=h_0h_2.
\end{equation}
%
If $h_2\not \equiv 0$, we multiply first equation in~\eqref{doppia} by $h_2^2$ and find, thanks to second equality in~\eqref{doppia}, 
$$
(h_1^2-h_2^2)(h_2^2+h_3^2)=2\frac a  b  h_1h_2(h_2^2+h_3^2).
$$
The facts that $h\notin \So_{I_0}(\Omega)$ and that $\Omega$ contains real points, ensure that $h_2^2+h_3^2\not\equiv0$.
Since $\So(\Omega)$ is an integral domain, we therefore have
$$h_1^2-2\frac a  b  h_1h_2-h_2^2\equiv0.$$

Now choose $x_0\in\Omega\cap\R$ such that $h_2(x_0)\neq0$; in a suitable neighborhood of $x_0$ the function $h_2$ is never-vanishing and hence we can consider the quotient $\tau=\frac{h_1}{h_2}$ which satisfies $\tau^2-2\frac a  b  \tau-1=0$.
Thus, as $a ^2+ b ^2=1$, we find $\tau=\frac{a \pm1}{ b }\neq 0$.
 So in the same neighbourhood we have that
$h_1=\frac{a \pm1}{ b }h_2\not\equiv0$ that is $h_2=-\frac{a \mp1}{ b }h_1$ since $a^{2}+b^{2}=1$. 
By the Identity Principle (see~\cite{G-S-St}, Theorem 1.12) we obtain that either $h_2\equiv-\frac{a +1}{ b }h_1$ or $h_2\equiv-\frac{a -1}{ b }h_1$.
 Using again the second equality in~\eqref{doppia} we find $h_1h_3=-h_0\frac{a \pm1}{ b }h_1$ and hence 
$h_3=-h_0\frac{a \pm1}{ b }$.

At last, we obtain 
\begin{equation*}\label{prima-formula}
\begin{split}
h&=h_0+h_1I_0-\frac{a \pm1}{ b }h_1J_0-h_0\frac{a \pm1}{ b }K_0=
h_0\left(1-\frac{a \pm1}{ b }K_0\right)
+h_1\left(I_0-\frac{a \pm1}{ b }K_0I_0\right)\\
&=
h_0\left(1-\frac{a \pm1}{ b }K_0\right)
+h_1\left(1-\frac{a \pm1}{ b }K_0\right)I_0=
\left(1-\frac{a \pm1}{ b }K_0\right)*(h_0
+h_1I_0).
\end{split}
\end{equation*}

If $h_2\equiv 0$, then $h_3\not\equiv0$, as $h\not\in\So_{I_0}(\Omega)$. 
The second equality in~\eqref{doppia} now becomes $h_{1}h_{3}\equiv 0$, which gives $h_{1}\equiv 0$. Thus
we obtain $h_{0}^{2}-h_{3}^{2}=2\frac{a}{b}h_{0}h_{3}$ and the same reasoning as above yields $h_{3}=-h_{0}\frac{a\pm 1}{b}$ that is $h=\left(1-\frac{a\pm1}{b}K_{0}\right)h_{0}$.
Thus, in both cases, setting $g=h_0+h_1I_0$, we have been able to find a function in $\So_{I_0}(\Omega)$ such that $h=\left(1-\frac{a \pm1}{ b }K_0\right)*g$.

Vice versa, suppose that $h=\left(1-\frac{a \pm1}{ b }K_0\right)*g$ for some $g\in\So_{I_0}(\Omega)$. Then 
$h*f*h^c=
\left(1-\frac{a \pm1}{ b }K_0\right)*g*f*g^c*\left(1-\frac{a \pm1}{ b }K_0\right)^c$.
Since $g*f*g^c$ belongs to $\So_{I_0}(\Omega)$ it is enough to show that for any $m=m_0+m_1I_0$ the function
$\left(1-\frac{a \pm1}{ b }K_0\right)*(m_0+m_1I_0)*\left(1-\frac{a \pm1}{ b }K_0\right)^c$ also belongs to 
$\So_{I_0}(\Omega)$. 
The vectorial part of the above function is given by $m_1\left(1-\frac{a \pm1}{ b }K_0\right)I_0\left(1+\frac{a \pm1}{ b }K_0\right)$,
so the following chain of equality gives the assertion,
\begin{align*}
\left(1-\frac{a \pm1}{ b }K_0\right)I_0\left(1+\frac{a \pm1}{ b }K_0\right)&=
\left(I_0-\frac{a \pm1}{ b }J_0\right)\left(1+\frac{a \pm1}{ b }K_0\right)\\
&=
I_0-\frac{a \pm1}{ b }J_0+\frac{a \pm1}{ b }I_0K_0-\left(\frac{a \pm1}{ b }\right)^2J_0K_0\\
&=
\left(1-\left(\frac{a \pm1}{ b }\right)^2\right)I_0-2\frac{a \pm1}{ b }J_0\\
&=
\frac{ b ^2-a ^2\mp2a  -1}{ b ^2}I_0-2\frac{a \pm1}{ b }J_0=
-2a \frac{a  \pm1}{ b ^2}I_0-2\frac{a \pm1}{ b }J_0\\
&=
-\frac{2(a \pm1)}{ b ^2}(a  I_0+ b  J_0)=-\frac{2(a \pm1)}{ b ^2}M_0.
\end{align*}

Now we turn to the case when $I_0$ and $M_0$ are linearly dependent, so that we can suppose $I_0=M_0$. Choosen any orthonormal basis $I_0,J_0,K_0$ of ${\rm Im}\HH$, we again use the expression for $h*f*h^c$ given by Equation~\eqref{coniugio} obtaining that  $h*f*h^c$ belongs to $\So_{I_0}(\Omega)$ if and only if 
$$
\begin{cases}
h_1h_2+h_0h_3=0,\\
h_1h_3-h_0h_2=0.
\end{cases}
$$
We multiply the first equation by $h_2$, the second one by $h_3$ and sum up, thus obtaining $h_1(h_2^2+h_3^2)\equiv0$. Again, the facts that $h\notin\So_{I_0}(\Omega)$ and that $\Omega$ contains real points, imply $h_1\equiv0$ and thus also $h_0$ has to be zero. This gives 
$$
h=h_2J_0+h_3K_0=h_2J_0-h_3J_0I_0=J_0*(h_2-h_3I_0);
$$
setting $g=h_2-h_3I_0$ gives the assertion. A direct inspection shows that all the functions of this form satisfy the condition $h*f*h^c\in\So_{I_0}(\Omega)$.
\end{proof}

\begin{remark}
If $I_0,J_0\in\SF$ with $I_0\perp J_0$ and $g\in\So_{I_0}(\Omega)$ then
$J_0*g=g^{c}*J_0$.
Thus the functions which appear in part $(ii)$ of the statement of the previous theorem can also be seen as products of a 
$\C_{I_0}$-preserving function for a suitable quaternion orthogonal to $I_0$.
\end{remark}

Now we turn to the ``dual'' problem, that is under which (non-trivial) conditions on $f$ the function $h*f*h^c$ is one-slice preserving when $h$ is.

\begin{teo}\label{secondo-coniugio}
Let $h=h_{0}+h_{1}I_{0}\in\So_{I_0}(\Omega)\setminus\So_\R(\Omega)$ and $f\in \So(\Omega)\setminus\So_{I_0}(\Omega)$. Then there exists $M_0\in\SF$ such that 
$h*f*h^c\in\So_{M_0}(\Omega)$ if and only if 
\begin{enumerate}[(i)]
\item in the case $I_0\perp M_0$, 
by setting $K_0=I_0M_0$, there exists $\rho\in\So_{\R}(\Omega)$ such that  
$(h^s)^2|\rho(h_0^2-h_1^2)$, $(h^s)^2|\rho h_0h_1$ and 
$$f=f_0+ \frac{\rho(h_0^2-h_1^2)}{(h^s)^2}M_0-2\frac{\rho h_0h_1}{(h^s)^2}K_0$$
\item in the case $\C_{I_{0}}\neq \C_{M_{0}}$ and $I_{0}\not\perp M_{0}$, 
by means of Notation~\ref{notation}, there exists $\rho\in\So_{\R}(\Omega)$ such that  
$h^s|\rho(h_0^2-h_1^2)$, $h^s|h_0h_1\rho$ and 
$$f=f_0+a \rho I_0+ \frac{ b \rho(h_0^2-h_1^2)}{h^s}J_0-2\frac{b\rho h_0h_1  }{h^s}K_0.$$
\end{enumerate} 
\end{teo}
\begin{proof}
We first consider the case in which $I_0$ and $M_0$ are linearly independent.
Using the basis given in Notation~\ref{notation}, we can decompose $f_v$ and $h_v$ as $f_1I_0+f_2 J_0+f_3K_0$ and $h_1I_0$ respectively. Since $f\notin \So_{I_0}(\Omega)$ then $f_2$ and $f_3$ are not both identically zero.
In order to apply Equality~\eqref{coniugio}, we need to compute 
$h_v\pv f_v=h_1(-f_3J_0+f_2K_0)$ and $(h_v\pv f_v)\pv h_v=h_1^2(f_2J_0+f_3K_0)$.
So we obtain
$$h*f*h^c=[h*f*h^c]_0+f_1h^sI_0+[f_2(h_0^2-h_1^2)-2h_0h_1f_3]J_0+[f_3(h_0^2-h_1^2)+2h_0h_1f_2]K_0,$$ 
where
$h^s=h_0^2+h_1^2\not\equiv0$.
Then the function $h*f*h^c$ is $\C_{M_0}$-preserving if and only if there exists $m_1\in\So_\R(\Omega)$ such that
$$
\begin{cases}	
f_1h^s=a  m_1,\\
f_2(h_0^2-h_1^2)-2h_0h_1f_3= b  m_1,\\
f_3(h_0^2-h_1^2)+2h_0h_1f_2=0.
\end{cases}
$$
Multiplying the second equation by $h_0^2-h_1^2$, the third one by $2h_0h_1$ and summing up, we obtain 
$[(h_0^2-h_1^2)^2+4h_0^2h_1^2]f_2= b  m_1(h_0^2-h_1^2)$ that can also be written as $(h^s)^2f_2= b  m_1 (h_0^2-h_1^2)$; analogously we find $(h^s)^2f_3=-2 b  m_1 h_0h_1$.

If $I_0$ and $M_0$ are orthogonal, then $a =0$, $ b =1$ and therefore 
$f_1\equiv 0$ and $(h^s)^2$ divides both  $m_1 (h_0^2-h_1^2)$ and $m_1 h_0h_1$. Setting  $\rho=m_1$ 
we can write $$f=f_0+ \frac{\rho(h_0^2-h_1^2)}{(h^s)^2}M_0-\frac{2\rho h_0h_1}{(h^s)^2}K_0$$
and we are done. 

If $I_0$ and $M_0$ are not orthogonal, then $f_1h^s=a  m_1$; since $a \neq0$ the function $h^s$ divides $m_1$ and hence there exists $\rho\in\So_\R(\Omega)$ such that $m_1=h^s\rho$. Thus $f_1=a  \rho$, moreover 
 $h^s$ divides both $\rho(h_0^2-h_1^2)$ and $h_0h_1\rho$, so that we have
$$f=f_0+a \rho I_0+ \frac{ b \rho(h_0^2-h_1^2)}{h^s}J_0-2\frac{b\rho h_0h_1  }{h^s}K_0.$$

In both cases, the converse is easily checked by direct inspection.

Now we are left to deal with the case $\C_{I_0}=\C_{M_0}$ showing that it cannot take place. Chosen any orthonormal basis $I_0,J_0,K_0$ of ${\rm Im}\HH$, we again compute $h*f*h^c$ by means of the decomposition in real and vectorial part as above obtaining that  $h*f*h^c$ belongs to $\So_{I_0}(\Omega)$ if and only if 
$$
\begin{cases}	
f_2(h_0^2-h_1^2)-2h_0h_1f_3=0,\\
f_3(h_0^2-h_1^2)+2h_0h_1f_2=0.
\end{cases}
$$
The same row reduction as above entails $f_2=f_3\equiv 0$ that contradicts the fact that $f\notin\So_{I_0}(\Omega)$. 
\end{proof}

The following three examples give general, explicit applications of the previous 
result, clarifying the role of the divisibility conditions. In particular Example~\ref{exnozeros} completely describes the case
when $h$ is never-vanishing and
Example~\ref{polynomials} does the same when $h$ is a polynomial.

\begin{example}\label{exnozeros}	
If $h$ is never-vanishing, then $h^{s}$ is never-vanishing and in particular it is invertible in $\So_\R(\Omega)$.  In this case the conditions on the divisibility by $h^s$ or $(h^s)^2$ are always satisfied, so if $I_0$ and $M_0$ are orthogonal then $f$ is given by
$$f=f_0+ \rho\left(\frac{h_0^2-h_1^2}{(h^s)^2}M_0-2\frac{h_0h_1}{(h^s)^2}K_0\right)$$
for any $f_0,\rho\in\So_\R(\Omega)$ and if $I_0$ and $M_0$ are lineraly independent but not orthogonal 
then $f$ is given by
$$f=f_0+\rho\left(a  I_0+ \frac{ b (h_0^2-h_1^2)}{h^s}J_0-2\frac{bh_0h_1  }{h^s}K_0\right)$$
for any $f_0,\rho\in\So_\R(\Omega)$.
\end{example}

\begin{example}	
If $h_0=h_1$, that is $h=h_{0}(1+I_{0}),$ then $h^s=2h_0^2$ automatically divides both $\rho(h_0^2-h_1^2)\equiv 0$ and  $h_0h_1\rho=h_0^2\rho$ for any $\rho\in\So_{\R}(\Omega)$. Thus if $I_0$ and $M_0$ are linearly independent and not orthogonal, then the divisibility conditions are always satisfied and hence 
$f=f_0+\rho\left(a  I_0- b  K_0\right)$
for any $f_0,\rho\in\So_\R(\Omega)$.
If $I_0$ and $M_0$ are orthogonal, the first divisibility condition is trivial and the second becomes $h_0^4|\rho h_0^2$ which is equivalent to $h_0^2|\rho$. In this case the function $\rho$ must be a multiple  of $h_0^2$, so there exists $\mu\in\So_\R(\Omega)$ such that $\rho=-2h_0^2 \mu$ and hence $f$ can be written as 
$f_0+\mu K_0$ for any $f_0,\mu\in\So_\R(\Omega)$. 
\end{example}

\begin{example}\label{polynomials}	
If $h_0,h_1$  are polynomials in $q$ (with real coefficients since they are slice preserving), then we can  factor out their GCD $\alpha$ and write them as $h_0=\alpha \beta_0$, $h_1=\alpha \beta_1$ with $\beta_0$ and $\beta_1$ coprime. Now,
$h^s=\alpha^2(\beta_0^2+\beta_1^2)$. If $I_0$ and $M_0$ are  linearly independent and not orthogonal, then  the divisibility conditions become 
$$\alpha^2(\beta_0^2+\beta_1^2)|\rho\alpha^2(\beta_0^2-\beta_1^2)\qquad \mbox{and} \qquad\alpha^2(\beta_0^2+\beta_1^2)|\rho \alpha^2\beta_0\beta_1$$ which are equivalent to 
$$(\beta_0^2+\beta_1^2)|\rho(\beta_0^2-\beta_1^2)\qquad \mbox{and} \qquad(\beta_0^2+\beta_1^2)|\rho\beta_0\beta_1.$$ As $\beta_0^2+\beta_1^2$ and $\beta_0^2-\beta_1^2$ are coprime in the ring $\R[q]$, 
then the ideal they generate in $\R[q]$ coincides with $\R[q]$. This entails that 
$\beta_0^2+\beta_1^2$ divides $\rho$ which therefore can be written as $(\beta_0^2+\beta_1^2)\mu$ for a suitable $\mu\in \So_{\R}(\Omega)$. The second divisibility condition now becomes trivial and hence we can write $f$ as 
\begin{align*}
f&=f_0+a (\beta_0^2+\beta_1^2)\mu I_0+  b \mu (\beta_0^2-\beta_1^2)J_0-2 b \mu \beta_0\beta_1K_0\\
&=
f_0+\mu\left(a (\beta_0^2+\beta_1^2)I_0+  b (\beta_0^2-\beta_1^2)J_0-2 b  \beta_0\beta_1K_0\right)
\end{align*}
for suitable $f_0,\mu\in\So_\R(\Omega)$.

If $I_0$ and $M_0$ are orthogonal, then  the divisibility conditions become 
$$\alpha^2(\beta_0^2+\beta_1^2)^2|\rho(\beta_0^2-\beta_1^2)\qquad \mbox{and} \qquad \alpha^2(\beta_0^2+\beta_1^2)^2|\rho \beta_0\beta_1.$$ Again $(\beta_0^2+\beta_1^2)^2$ and $\beta_0^2-\beta_1^2$ are coprime in the ring $\R[q]$, 
then $(\beta_0^2+\beta_1^2)^2$ divides $\rho$ which therefore can be written as $(\beta_0^2+\beta_1^2)^2\mu$ for a suitable $\mu\in \So_{\R}(\Omega)$. 
The above relations become 
$$\alpha^2|\mu(\beta_0^2-\beta_1^2)\qquad \mbox{and} \qquad \alpha^2|\mu \beta_0\beta_1.$$ 
Since $\beta_0^2-\beta_1^2$ and $\beta_0\beta_1$ are coprime in the ring $\R[q]$, 
then the ideal generated by $\mu(\beta_0^2-\beta_1^2)$ and $\mu \beta_0\beta_1$ in $\So_\R(\Omega)$ coincides with the ideal generated by $\mu$ and therefore $\alpha^2$ divides $\mu$ in $\So_\R(\Omega)$. Thus we can write $\mu=\alpha^2\nu$ for a suitable slice preserving function $\nu$.
Hence we can write $f$ as 
\begin{align*}
f&=f_0+ \frac{\alpha^2\nu (\beta_0^2+\beta_1^2)^2 \alpha^2(\beta_0^2-\beta_1^2)}{\alpha^4(\beta_0^2+\beta_1^2)^2}M_0-2\frac{\alpha^2\nu (\beta_0^2+\beta_1^2)^2 \alpha^2\beta_0\beta_1}{\alpha^4(\beta_0^2+\beta_1^2)^2}K_0\\
&=f_0+ \nu \left( (\beta_0^2-\beta_1^2)M_0-2 \beta_0\beta_1K_0\right)
\end{align*}
for suitable $f_0,\nu\in\So_\R(\Omega)$. As $(\beta_0^2-\beta_1^2)M_0-2 \beta_0\beta_1K_0= (\beta_{0}-\beta_{1}I_{0})^{2}M_{0}=\left(\frac{h^{c}}{\alpha}\right)^{2}M_{0}$, last equality can also be written as
$$
f=f_{0}+\nu\left(\frac{h^{c}}{\alpha}\right)^{2}M_{0}
$$
for some $f_0,\nu\in\So_\R(\Omega)$.
\end{example}

We now change slightly our point of view by studying the existence of a one-slice preserving solution of the equation $h*f*h^c=g$, given the function $g$ and one between $f$ and $h$ which is chosen to be one-slice preserving.
In Theorem~\ref{coniugato1}, we answer this question when $f$  is one-slice preserving  and $g$ is given: we look for the solvability of $g=h*f*h^{c}$, where the solution $h$ should be again one-slice preserving; Theorem~\ref{coniugato2} answers the same issue exchanging the role of $f$ and $h$.

From now until the end of this section we ask that $\Omega\subset\HH$ is such that $\Omega_{I_{0}}=\Omega\cap\C_{I_{0}}$ has trivial first fundamental group, i.e.: $\pi_{1}(\Omega_{I_{0}})=0$ for one, and then any, $I_0\in\SF$.

Given $f\in\So_{I_{0}}(\Omega)$, a slice $\C_{M_{0}}$ with $M_{0}\in\SF$, and $g\in\So(\Omega)$, 
if $I_{0}, M_{0}$ are linearly independent 
we follow Notation~\ref{notation} writing $f=f_{0}+f_{1}I_{0}$, and
$g=g_{0}+g_{1}I_{0}+g_{2}J_{0}+g_{3}K_{0}$, with $f_{0}, f_{1}, g_{0},g_{1},g_{2},g_{3}\in\So_{\R}(\Omega)$.
%


\begin{teo}\label{coniugato1}
Given $f$, $M_{0}$ and $g$ as above, there
exists $h\in\So_{M_{0}}(\Omega)$, such that $g=h*f*h^{c}$ if and only if
\begin{enumerate}[(i)]
\item in the case $\C_{M_{0}}=\C_{I_{0}}$, then $g\in\So_{I_{0}}(\Omega)$, $f$ divides $g$, the quotient $g/f$ belongs to $\So_{\R}(\Omega)$, it is non-negative over the reals and all its real zeros have even multiplicity;
\item in the case $\C_{M_{0}}\neq\C_{I_{0}}$ and $M_{0}\not\perp I_{0}$, then $f_{0}$
divides $g_{0}$ and  $f_{1}$ divides $g_{v}$; denoted by $\alpha_{0}$ the quotient $g_{0}/f_{0}$ and $\alpha_{l}$ the quotient $g_{l}/f_{1}$, for $l=1,2,3$, we have that
$\alpha_{0}$ and $\alpha_{2}$ are non-negative over the reals and their real zeros have even multiplicity; moreover $\alpha_{2}$ has spherical zeroes of multiplicity
which is a multiple of $4$; the multiplicities of the zeroes of $\alpha_{3}$ are at least
equal to half the multiplicities of the zeroes of $\alpha_{2}$ and
\begin{equation}\label{condizione1}
a(\alpha_{0}-\alpha_{1})=b\alpha_{2},
\end{equation}
\begin{equation}\label{condizione2}
2ab\alpha_{1}\alpha_{2}= a^{2}\alpha_{3}^{2}+(2a^{2}-1)\alpha_{2}^{2}.
\end{equation}
\item in the case $M_{0}\perp I_{0}$, then $f_{0}$
divides $g_{0}$, the quotient $\alpha_{0}=g_{0}/f_{0}$ is non-negative over the reals and its real zeros have even multiplicity; $g_{2}\equiv 0$; $f_{1}$ divides $g_{v}$ and denoted by $\alpha_{l}$ the quotient $g_{l}/f_{1}$, for $l=1,3$, we have
\begin{equation}\label{ternapitagorica}
\alpha_{0}^{2}=\alpha_{1}^{2}+\alpha_{3}^{2}.
\end{equation}
\end{enumerate}
\end{teo}

\begin{proof}
\textit{(i)} In this case $h*f*h^{c}=h^{s}f$. 
Thus $g\in\So_{I_{0}}(\Omega)$, the function $f$ divides $g$ and the quotient $g/f$ belongs to $\So_{\R}(\Omega)$, it is non-negative over the reals and has real zeroes
of even multiplicities. Vice versa if $g\in\So_{I_{0}}(\Omega)$ can be written as 
$g=\alpha f$, for $\alpha\in\So_{\R}(\Omega)$ with $\alpha$ non-negative over the reals and having real zeroes of even multiplicity, thanks to Theorem~\ref{radicesimmetrizzata}, we can find $h\in\So_{I_{0}}(\Omega)$ such that
$h^{s}=\alpha$ and hence $g=h^{s}f=h*f*h^{c}$.\\
If $\C_{I_{0}}\neq \C_{M_{0}}$ we use Notation~\ref{notation} and
we write $h=h_{0}+h_{1}M_{0}$ for suitable $h_{0},h_{1}\in\So_{\R}(\Omega)$.
According to~\eqref{coniugio} we have
$$
h*f*h^{c}=f_{0}(h_{0}^{2}+h_{1}^{2})+f_{1}[(h_{0}^{2}+(2a^{2}-1)h_{1}^{2})I_{0}+2abh_{1}^{2}J_{0}-2bh_{0}h_{1}K_{0}].
$$
Then $h*f*h^{c}=g$ is equivalent to the following system of equations:
\begin{equation}\label{condizionig}
\begin{cases}
g_{0}=f_{0}(h_{0}^{2}+h_{1}^{2})\\
g_{1}=f_{1}(h_{0}^{2}+(2a^{2}-1)h_{1}^{2})\\
g_{2}=2abf_{1}h_{1}^{2}\\
g_{3}=-2bf_{1}h_{0}h_{1}.
\end{cases}
\end{equation}
\textit{(ii)} In this case both $a$ and $b$ are positive, so if $g=h*f*h^{c}$, we have~\eqref{condizionig} and the necessity is straightforward. Vice versa the conditions on $\alpha_{2}$ guarantee the existence of $h_{1}\in\So_{\R}(\Omega)$ such that $\alpha_{2}=2abh_{1}^{2}$; in particular
 the multiplicities of the zeroes of $h_{1}$ are equal to half of the multiplicities of the 
 zeroes of $\alpha_{2}$. The relations on the multiplicities of the zeroes of
 $\alpha_{2}$ and $\alpha_{3}$ allow us to find $h_{0}=\frac{\alpha_{3}}{-2bh_{1}}\in\So_{\R}(\Omega)$. 
Equality~\eqref{condizione2} gives that 
$$\alpha_{1}=\frac{a^{2}\alpha_{3}^{2}+(2a^{2}-1)\alpha_{2}^{2}}{2ab\alpha_{2}}.$$
Substituting the formulas for $\alpha_{2}$ and $\alpha_{3}$ in terms of $h_{0}$ and 
$h_{1}$ in the previous equality we find $\alpha_{1}=h_{0}^{2}+(2a^{2}-1)h_{1}^{2}$.
Now~\eqref{condizione1} gives $\alpha_{0}=\alpha_{1}+\frac{b}{a}\alpha_{2}$. By substituting $\alpha_{1}$ and $\alpha_{2}$ in terms of $h_{0}$ and
$h_{1}$ and recalling that $a^{2}+b^{2}=1$ we have $\alpha_{0}=h_{0}^{2}+h_{1}^{2}$ and thus we find $h=h_{0}+h_{1}M_{0}$ such that $h*f*h^{c}=g$.\\
\textit{(iii)} In this case $a=0$ and $b=1$, so if $g=h*f*h^{c}$, we have $g_{2}\equiv 0$ and
\begin{equation*}
\begin{cases}
g_{0}=f_{0}(h_{0}^{2}+h_{1}^{2})\\
g_{1}=f_{1}(h_{0}^{2}-h_{1}^{2})\\
g_{3}=-2f_{1}h_{0}h_{1}
\end{cases}
\end{equation*}
and the necessity of the conditions is trivial. Vice versa suppose that $f_{0}$ divides $g_{0}$,
the quotient $\alpha_{0}=g_{0}/f_{0}$ is non-negative over the reals and has real zeros of even multiplicity; $g_{2}\equiv 0$; $f_{1}$ divides $g_{v}$ and, by denoting by $\alpha_{l}$ the quotient $g_{l}/f_{1}$, for $l=1,3$, Equality~\eqref{ternapitagorica} is satisfied.
Thanks to Proposition~\ref{nostroweierstrass}
we can get rid of the common spherical zeroes of $\alpha_{0}$ and $\alpha_{1}$ and so of $\alpha_{3}$.
Now the hypothesis on the sign of $\alpha_{0}$ over the reals, together with~\eqref{ternapitagorica}, entail that $u=\frac{\alpha_{0}+\alpha_{1}}{2}$ and $v=\frac{\alpha_{0}-\alpha_{1}}{2}$ 
both belong to $\So_{\R}(\Omega)$, are non-negative over the reals and
thus their real zeroes must have even multiplicity. 
Since  $\alpha_{0}$ and $\alpha_{1}$ have no common spherical
zeroes, then the same holds for $u$ and $v$.
If  $S_{0}$ is a spherical zero for $u$, then it is also a spherical zero for $\alpha_{3}^{2}$ and hence for $\alpha_{3}$. The fact that it is not a spherical 
zero for $v$ ensures that it is a spherical zero of $u$ with multiplicity which is a multiple
of $4$; the same property on the multiplicities of spherical zeroes holds for $v$.
Thus, thanks to Proposition 3.1 in~\cite{A-dF},  we can find $h_{0},h_{1}\in\So_{\R}(\Omega)$, such that $u=h_{0}^{2}$ and $v=h_{1}^{2}$.
The definitions of $u$ and $v$ entail that $\alpha_{0}=h_{0}^{2}+h_{1}^{2}$ and
$\alpha_{1}=h_{0}^{2}-h_{1}^{2}$. Thus~\eqref{ternapitagorica} gives $\alpha_{3}^{2}=4h_{0}^{2}h_{1}^{2}$; up to a change of sign for $h_{1}$ we
obtain $\alpha_{3}=-2h_{0}h_{1}$, that is $g=h*f*h^{c}$ for $h=h_{0}+h_{1}M_{0}$.
\end{proof}

Now we deal with the ``dual'' problem of giving appropriate conditions on a one-slice preserving function $h$
and a function $g$ in order to find a one-slice preserving map $f$ such that $h*f*h^{c}=g$.  
Given $h\in\So_{M_{0}}(\Omega)\setminus\So_{\R}(\Omega)$, $\C_{I_{0}}$ a slice in $\HH$ with $I_{0}\in\SF$ and $g\in\So(\Omega)$, if $I_{0}$ and $M_{0}$ are linearly independent,
we choose an orthonormal basis $I_0,J_0,K_0$ of ${\rm Im}\HH$ as in Notation~\ref{notation}.
Moreover, we write $h=h_{0}+h_{1}M_{0}$, and
$g=g_{0}+g_{1}I_{0}+g_{2}J_{0}+g_{3}K_{0}$, with $h_{0}, h_{1}$ and
$g_{0},g_{1},g_{2},g_{3}$ belonging to $\So_{\R}(\Omega)$.

\begin{prop}\label{coniugato2}
Given $h$, $I_{0}$ and $g$ as above, there
exists $f\in\So_{I_{0}}(\Omega)$, such that $g=h*f*h^{c}$ if and only if
\begin{enumerate}[(i)]
\item in the case $\C_{M_{0}}=\C_{I_{0}}$, we have $g\in\So_{I_{0}}(\Omega)$ and $h^{s}$ divides $g$.
\item in the case $\C_{M_{0}}\neq\C_{I_{0}}$ and $M_{0}\not\perp I_{0}$, we have that $h^{s}$
divides $g_{0}$, $h_{0}^{2}+(2a^{2}-1)h_{1}^{2}$ divides $g_{1}$,  $h_{1}^{2}$ divides $g_{2}$, $h_{0}h_{1}$ divides $g_{3}$ and
\begin{equation}\label{condizione3}
h_{0}g_{2}+ah_{1}g_{3}=0,
\end{equation}
\begin{equation}\label{condizione4}
(h_{0}^{2}+(2a^{2}-1)h_{1}^{2})g_{2}=2abh_{1}^{2}g_{1}.
\end{equation}
\item  in the case $M_{0}\perp I_{0}$, we have that $h^{s}$ divides $g_{0}$, $h_{0}^{2}-h_{1}^{2}$ divides $g_{1}$, $h_{0}h_{1}$ divides $g_{3}$, $g_{2}\equiv 0$ 
and
\begin{equation}\label{condizione5}
2h_{0}h_{1}g_{1}+(h_{0}^{2}-h_{1}^{2})g_{3}=0.
\end{equation}
\end{enumerate}
\end{prop}

\begin{proof}
\textit{(i)} In this case $g=h^{s}f$, therefore the necessity of the conditions holds
trivially. Vice versa if $g$ belongs to $\So_{I_{0}}(\Omega)$ and $h^{s}$ divides $g$
then the quotient $g/h^{s}$ is in $\So_{I_{0}}(\Omega)$ because
$h^{s}$ is slice preserving. The thesis is obtained by taking $f=g/h^{s}$.\\
If $\C_{I_{0}}\neq\C_{M_{0}}$,
we write $f=f_{0}+f_{1}I_{0}$ for suitable $f_{0},f_{1}\in\So_{\R}(\Omega)$.
The computations performed in the proof of Theorem~\ref{coniugato1} entail
the system of conditions~\eqref{condizionig}.\\
\textit{(ii)} In this case the necessity of conditions is again trivial from 
system~\eqref{condizionig}. Vice versa, setting $f_{0}=g_{0}/h^{s}$, $f_{1}=g_{2}/(2abh_{1}^{2})$, we obtain, thanks to~\eqref{condizione3} and~\eqref{condizione4}
that the equality $h*f*h^{c}=g$ holds thanks to~\eqref{condizionig}.\\
\textit{(iii)} Again, the necessity of the conditions is straightforward. 
If $h_{0}\not\equiv 0$, then setting $f_{0}=g_{0}/h^{s}$ and $f_{1}=g_{3}/(-2h_{0}h_{1})$ gives the thesis, thanks to~\eqref{condizione5} and~\eqref{condizionig}.
If $h_{0}\equiv 0$ we then have $h^{s}=h_{1}^{2}$ and~\eqref{condizione5} entails $g_{3}\equiv 0$. Then, setting
$f_{0}=g_{0}/h_{1}^{2}$ and $f_{1}=-g_{1}/h_{1}^{2}$ ends the proof again thanks to~\eqref{condizionig}.
\end{proof}

\begin{remark}
The fact that statement and proof of Proposition~\ref{coniugato2} are neater than
the ones of Theorem~\ref{coniugato1} can be seen as a consequence that
in a certain sense the equality $h*f*h^{c}=g$ is ``linear'' in $f$ and ``quadratic'' in $h$. 
\end{remark}

\section{More products}
Given $f$ and $h$ such that $f*h$ is one-slice preserving, it is not always true that $h*f$ is one-slice preserving
as well.
The results obtained so far for the conjugate of a given function give us a better
understanding of the behaviour of the $*$-product. In particular
we are able to give necessary and sufficient conditions on the two factors in
order that the two $*$-products in different orders are both one-slice preserving. 
The first result explicitly describes the two factors in terms of functions which
are one-slice preserving, showing that if the products of two functions in the two possible orders are both one-slice preserving, then the two factors are obtained by suitably ``twisting'' two one-slice preserving function which preserve the same slice for a fixed quaternion.

\begin{teo}\label{moreproducts}
Let $f,h\in\So(\Omega)\setminus\{0\}$. There exist $I_{0}, M_{0}\in\SF$ such that
$f*h\in\So_{I_{0}}(\Omega)\setminus\So_{\R}(\Omega)$ and $h*f\in\So_{M_{0}}(\Omega)\setminus\So_{\R}(\Omega)$ if and only if
\begin{enumerate}[(i)]
\item in the case $\C_{I_{0}}=\C_{M_{0}}$, either $f,h\in\So_{I_{0}}(\Omega)$ or there exist $J_{0},K_{0}\in\SF$ both orthogonal to $I_{0}$
and $\tilde f, \tilde h\in\So_{I_{0}}(\Omega)\setminus\{0\}$ such that
\begin{equation}\label{biprodottodip}
f=\tilde f*K_{0},\qquad
h=J_{0}*\tilde h.
\end{equation}
\item in the case $\C_{I_{0}}\neq\C_{M_{0}}$, there exist $\tilde f, \tilde h\in\So_{I_{0}}(\Omega)\setminus\{0\}$, such that 
\begin{equation}\label{biprodotto}
f=\tilde f*\left(1+\frac{a \pm1}{ b }K_0\right),\qquad
h=\left(1-\frac{a \pm1}{ b }K_0\right)*\tilde h,
\end{equation}
where we follow Notation~\ref{notation}.
\end{enumerate}
\end{teo}
  
\begin{proof}
First of all we notice that 
\begin{equation} \label{prodconj}
\begin{split}
h*(f*h)*h^{c}&=h*f*h^{s}=h^{s}(h*f),\\
f^{c}*(f*h)*f&=f^{s}(h*f).
\end{split}
\end{equation}

As $f*h\in\So_{I_{0}}(\Omega)\setminus\So_{\R}(\Omega)$ and $h*f\in\So_{M_{0}}(\Omega)\setminus\So_{\R}(\Omega)$, Theorem~\ref{primo-coniugio} drives the remainder of the proof.\\
\textit{(i)} For the sufficiency of the conditions it is enough to perform direct computations keeping in mind that:
\begin{itemize}
\item if $J_{0},K_{0}\in\SF$ are both orthogonal to $I_{0}$, then $J_{0}K_{0}\in\C_{I_{0}}$;
\item if $g\in\So_{I_{0}}(\Omega)$ and $J_{0}\in\SF$ is orthogonal to $I_{0}$, then
$J_{0}*g=g^{c}*J_{0}$.
\end{itemize}
Vice versa, if at least one between $f$ and $h$ does not belong to $\So_{I_{0}}(\Omega)$, then, both $f$ and
$h$ do not belong to $\So_{I_{0}}(\Omega)$ because both their $*$-products do.
Thus Theorem~\ref{primo-coniugio} allows us to find $J_{0},K_{0}\in\SF$ both orthogonal to $I_{0}$ and $g, \tilde h\in\So_{I_{0}}(\Omega)$ such that 
\begin{equation*}
f^{c}=K_{0}*g,\qquad
h=J_{0}*\tilde h.
\end{equation*}
Setting $\tilde f=-g^{c}$ gives~\eqref{biprodottodip}.
\\
\textit{(ii)} Again, the sufficiency of the conditions is proved by direct inspection. Vice versa, 
we observe that by first equality in~\eqref{prodconj}, $h$ cannot belong to $\So_{I_{0}}(\Omega)$. Then, by adopting Notation~\ref{notation}, we have that Theorem~\ref{primo-coniugio} ensures the existence of $g, \tilde h\in\So_{I_{0}}(\Omega)$ such that
\begin{equation*}
f^{c}=\left(1-\frac{a \pm1}{ b }K_0\right)*g,\qquad
h=\left(1-\frac{a \pm1}{ b }K_0\right)*\tilde h.
\end{equation*}
Setting $\tilde f=g^{c}$ gives~\eqref{biprodotto}.
\end{proof}

In the case one of the two factors appearing in the previous result is one-slice preserving itself,
the special form of the two factors obtained in the statement becomes even more special.

\begin{prop}\label{lastprod}
Let $f\in\So_{I_0}(\Omega)\setminus\So_{\R}(\Omega)$ and $h\in\So(\Omega)\setminus\{0\}$ such that
$f*h\in\So_{M_0}(\Omega)$ and $h*f\in\So_{N_0}(\Omega)$, then
\begin{enumerate}[(i)]
\item in the case $\C_{M_{0}}=\C_{N_{0}}$, either $\C_{I_0}=\C_{M_0}$ and $f,h\in\So_{M_0}(\Omega)$ or $I_0\bot M_0$ and
there exist $\alpha\in\So_{\R}(\Omega)$, $\tilde{h}\in\So_{M_0}(\Omega)$ and $J_0\bot M_0$ such that $f=\alpha I_0$ and $h=J_0*\tilde{h}$;
\item in the case $\C_{M_{0}}\neq\C_{N_{0}}$, setting $K_0=\frac{M_0\wedge N_0}{|M_0\wedge N_0|}$, 
$L_0=-M_0 K_0$ (so that $M_0, L_0, K_0$ is a positive orthonormal basis) and $N_0=aM_0+bL_0$, we can write $I_0=lM_0+mL_0+nK_0$ with 
$bm+l(a\pm1)=0$ and there exist
$\alpha\in\So_{\R}(\Omega)$, $\tilde{h}\in\So_{M_0}(\Omega)$ such that 
\begin{equation*}
f=\alpha\left(-\frac{n}{m}+M_0\right)*\left(1+\frac{a \pm1}{ b }K_0\right),\qquad
h=\left(1-\frac{a \pm1}{ b }K_0\right)*\tilde h.
\end{equation*}
\end{enumerate}
\end{prop}

\begin{proof}
\textit{(i)} In this case Theorem~\ref{moreproducts} entails that either $f,h\in \So_{M_0}(\Omega)$ so that $\C_{I_0}=\C_{M_0}$
or there exist $J_{0},K_{0}\in\SF$ both orthogonal to $M_{0}$ and $\tilde f, \tilde h\in\So_{M_{0}}(\Omega)\setminus\{0\}$ such that
$f=\tilde f*K_{0}$ and $h=J_{0}*\tilde h$. Setting $\tilde f=\tilde f_0+\tilde f_1 M_0$,
a trivial computation shows that $f=\tilde f_0 K_0+\tilde f_1 M_0 K_0$.
Now consider the orthonormal basis $M_{0},K_{0},L_{0}=M_{0}K_{0}$ of ${\rm Im}\HH$ and write
$I_{0}=aM_{0}+bK_{0}+cL_{0}$. Since $f\in\So_{I_{0}}(\Omega)$ there exist $\alpha_{0},\alpha_{1}\in\So_{\R}(\Omega)$ such that $f=\tilde{f_{0}}K_{0}+\tilde{f_{1}}L_{0}=\alpha_{0}+\alpha_{1}I_{0}=\alpha_{0}+a\alpha_{1}M_{0}+b\alpha_{1}K_{0}+c\alpha_{1}L_{0}$. Uniqueness given by 
Proposition~\ref{GMP} entails $\alpha_{0}=0$, $a=0$, $\tilde{f_{0}}=b\alpha_{1}$, $\tilde{f_{1}}=c\alpha_{1}$, so that $I_{0}\perp M_{0}$ and $f=\alpha_{1}I_{0}$.
\\
\textit{(ii)} Again by Theorem~\ref{moreproducts} there exist $\tilde f, \tilde h\in\So_{M_{0}}(\Omega)\setminus\{0\}$, such that 
\begin{equation*}
f=\tilde f*\left(1+\frac{a \pm1}{ b }K_0\right),\qquad
h=\left(1-\frac{a \pm1}{ b }K_0\right)*\tilde h.
\end{equation*}
Setting $\tilde f=\tilde f_0+\tilde f_1 M_0$
we find that $f_0=\tilde f_0$ and 
$$f_v=\tilde f_1 M_0+\tilde f_0 \frac{a\pm 1}{b}K_0+\tilde f_1\frac{a\pm 1}{b}M_0K_0=
\tilde f_1 M_0-\tilde f_1\frac{a\pm 1}{b}L_0+\tilde f_0 \frac{a\pm 1}{b}K_0.$$
As $f\in\So_{I_0}(\Omega)\setminus\So_{\R}(\Omega)\setminus\{0\}$ there exists $\gamma\in\So_{\R}(\Omega)$ such that
$f_v=\gamma I_0$ which is equivalent to
\begin{equation}\label{moreprodcase2}
\begin{cases}
\tilde f_1=\gamma l\\
-\frac{a\pm1}{b}\tilde f_1=\gamma m\\
\frac{a\pm 1}{b}\tilde f_0=\gamma n
\end{cases}
\end{equation}
where $I_0=lM_0+mL_0+nK_0$. Since $\tilde f_1\not\equiv 0$ the comparison between the first two equations in~\eqref{moreprodcase2}
gives $bm+l(a\pm1)=0$,
while last two equations entail $\tilde f_0=-\frac{n}{m}\tilde f_1$, that is 
$$
\tilde f=-\frac{n}{m}\tilde f_1+\tilde f_1 M_0=\tilde f_1\left(-\frac{n}{m}+M_0\right).
$$
Setting $\alpha=\tilde f_1$ ends the proof.
\end{proof}

\begin{remark}
We point out that, again
by direct computation, we have that conditions \textit{(i)} and \textit{(ii)} in Proposition~\ref{lastprod} are also sufficient in order to obtain $f*h\in\So_{M_0}(\Omega)$ and $h*f\in\So_{N_0}(\Omega)$.
\end{remark}

\section{$*$-Powers}

In order to conclude our investigation on the structure of one-slice preserving functions, we turn our attention to the problem of classifying slice regular functions whose $*$-powers preserve 
one single slice or all of them. To rule out trivial cases, in this section we will always
consider $f\in\So(\Omega)\setminus\So_{\R}(\Omega)$, which means that the vectorial part of $f$ is not identically zero.

The first tool we need is the following computation of the $*$-powers of $f$ in
terms of the components of the splitting $f=f_{0}+f_{v}$.

\begin{lemma}\label{conto-potenza}
Let $f=f_0+f_v\in\So(\Omega)$, then 
$$
f^{*d}=\sum_{n=0}^{[d/2]}(-1)^n\binom{d}{2n}f_0^{d-2n}(f_v^s)^{n}+
\left(\sum_{n=0}^{[(d-1)/2]}(-1)^n\binom{d}{2n+1}f_0^{d-(2n+1)}(f_v^s)^{n}\right)f_v.
$$
\end{lemma}

\begin{proof}
Since $f_v*f_v=-f_v*f_v^c=-f_v^s$, we have 
\begin{align*}
f^{*d}&=(f_0+f_v)^{*d}=\sum_{m=0}^d\binom{d}{m}f_0^{d-m}(f_v)^{*m}\\
&=\sum_{n=0}^{[d/2]}(-1)^n\binom{d}{2n}f_0^{d-2n}(f_v^s)^{n}+
\left(\sum_{n=0}^{[(d-1)/2]}(-1)^n\binom{d}{2n+1}f_0^{d-(2n+1)}(f_v^s)^{n}\right)f_v.
\end{align*}
\end{proof}

\begin{remark}
Notice that the above computation for $d=2$ entails that
$f^{*2}\in\So_{\R}(\Omega)$ if and only if $f_{0}\equiv 0$. 
\end{remark}

Because of the above remark, from now on we take into account only functions
$f=f_{0}+f_{v}$ with $f_{0}\not\equiv 0$.
Moreover, in order to avoid trivial statements, if we are looking for 
a suitable $*$-power $d>2$ of $f$ which is one-slice preserving,
we rule out the case in which $f$ itself is one-slice preserving. In particular, thanks to Lemma~\ref{conto-potenza}, under this hypothesis we have the following
\begin{remark}
Let $f=f_0+f_v\in \So(\Omega)$ which preserves no slice. Then
  $f^{*d}$ is one-slice preserving if and only if $f^{*d}$ is slice preserving.
This last condition
is equivalent to
\begin{equation}\label{power}
\sum_{n=0}^{[(d-1)/2]}(-1)^n\binom{d}{2n+1}f_0^{d-(2n+1)}(f_v^s)^{n}\equiv 0.
\end{equation}
\end{remark}
To simplify the statement of the results, from now on we only consider functions $f=f_0+f_v\in \So(\Omega)$ such that $f$ preserves no slice.
In order to carry on our investigation we need to set some notation and 
quote a result on the real roots of a binary form.

We denote by $Q_d(x,y)$ the homogeneous polynomial of degree $d$ given by  
$$Q_d(x,y)=\sum_{n=0}^{[(d-1)/2]}(-1)^n\binom{d}{2n+1}x^{d-(2n+1)}y^{2n+1}.$$
We denote by $\Sigma_d\subset{\mathbb P}(\R^2)\sim\R\cup\{\infty\}$ the set of roots of $Q_d$ different from $0$ and $\infty$; a straightforward computation shows that $0=[0:1]$ is a root of $Q_d$ and only if $d$ is even and that $\infty=[1:0]$ is  always a root of $Q_d$.
Due to Proposition 41 in~\cite{M}, see also~\cite{C-R}, $Q_d$ has $d$ real distinct roots and therefore $\Sigma_d$ contains $d-2$ elements if $d$ is even and $d-1$ if $d$ is odd. 

Now choose $q_0\in\Omega\cap\R$ such that $f_0(q_0)\neq 0$ and $f_v(q_0)\neq 0$. Then we can choose a suitable spherical neighborhood $U=B_{q_0}(r)$ of the point $q_0$ where both $f_0$ and $f_v$ never vanish, which entails that also $f_v^s$ is never-vanishing on $U$. Thus Corollary 3.2 in~\cite{A-dF} gives the existence of $\rho\in\So_\R(U)$ such that 
$\rho^2=f_v^s$ on $U$.   

Equality~\eqref{power} thus implies that on $U$ we have
\begin{equation*}
\sum_{n=0}^{[(d-1)/2]}(-1)^n\binom{d}{2n+1}f_0^{d-(2n+1)}\rho^{2n+1}
\equiv 0.
\end{equation*}

In particular, at any $q\in U$ the point $[f_0(q):\rho(q)]\in\R\setminus\{0\}$ is a root of $Q_d$ and thus there exists $\xi(q)\in \Sigma_d$ such that $[f_0(q):\rho(q)]=\xi(q)$. As $U$ is connected, $\Sigma_d$ is a finite subset of $\R$ and the map $U\ni q\mapsto (f_0(q),\rho(q))\in \R^2\setminus\{(0,0)\}$ is continuous, then $\xi(q)$ is constant and hence $f_0=\xi \rho$ on $U$. The Identity Principle then entails that $f_0^2=\xi^2 \rho^2=\xi^2 f_v^s$ on all $\Omega$.
The above argument can be summarized as follows:
\begin{prop}\label{potenza-d-esima}
Let $f=f_{0}+f_{v}\in \So(\Omega)$.
There exists $d>2$ such that $f^{*d}$ is slice preserving
if and only if there exists $\xi\in\Sigma_d$ 
such that $f_{0}^{2}\equiv \xi^2 f_{v}^{s}$,
that is $f_{0}$ is a square root of $\xi^2 f_{v}^{s}$.
\end{prop}

If the  first fundamental group of $\Omega_I=\Omega\cap \C_I$ is trivial, then the set of slice regular functions 
whose $d-th$ $*$-power belongs to $\So_\R(\Omega)$ can be characterized even more precisely.
Indeed we have 
\begin{cor}
Let $f=f_{0}+f_{v}\in \So(\Omega)$ and suppose $\pi_1(\Omega_I)=\{0\}$ for some $I\in\SF$.
There exists $d>2$ such that $f^{*d}$ belongs to $\So_{\R}(\Omega)$
if and only if the zero set of $f_v$ does not contain  non real isolated zeroes of odd multiplicities and
there exists $\xi\in\Sigma_d$ 
such that $f_{0}$ is a square root of $\xi^2 f_{v}^{s}$.
\end{cor}
\begin{proof}
By Proposition~\ref{potenza-d-esima} there exists $\xi\in\Sigma_d$ 
such that $f_{0}^{2}\equiv \xi^2 f_{v}^{s}$ which is equivalent to $f_v^s=\frac{f_0^2}{\xi^2}$. 
Then the functions $f_v^s$ has a square root $\frac{f_0}{\xi}\in\So_\R(\Omega)$; 
the hypothesis on the first fundamental group of $\Omega_I$ together with Corollary~3.2 in~\cite{A-dF}, entail that
this is equivalent
to the fact that the zero set of $f_v$  does not contain  non real isolated zeroes of odd multiplicities. 
\end{proof}
The following example contains the explicit expressions of $Q_d$ and $\Sigma_d$ for $d=3,\dots, 10$.
\begin{example}\label{exbinforms}
$\phantom{}$
\\
\begin{center}
\begin{tabular}{| l || l | p{6.5cm} |}
\hline
$d$ & $Q_d(x,y)$ & $\Sigma_d$ \\ \hline \hline
$3$ & $3x^2y-y^3$ & $\left\{\pm\frac{\sqrt3}3\right\}$ \\ \hline
$4$ & $4x^3y-4xy^3 $ & $\left\{ \pm1\right\}$ \\ \hline
$5$ &$5x^4y-10x^2y^3+y^5$ & $\left\{ \pm\frac{\sqrt{25\pm 10\sqrt5}}{5}\right\}$\\ \hline
$6$ &$6x^5y-20x^3y^3+6xy^5$ &$\left\{ \pm\frac{\sqrt{3}}{3}, \pm \sqrt3\right\}$ \\ \hline
\multirow{3}{*}{$7$} & \multirow{3}{*}{$7x^6y-35x^4y^3+21x^2y^5-y^7$} & $\left\{ \pm
\left(\frac{1}{3} \left(5+8 \cos \left(\frac{1}{3} \tan ^{-1}\left(\frac{3
   \sqrt{3}}{13}\right)\right)\right)\right)^{\frac{1}{2}},\right.$\\
   & & $\pm \left(\frac{1}{3} \left(5\pm4 \sqrt{3} \sin \left(\frac{1}{3} \tan ^{-1}\left(\frac{3
   \sqrt{3}}{13}\right)\right)+\right.\right.$\\
& &   $\left.\left.\left.-4 \cos \left(\frac{1}{3} \tan ^{-1}\left(\frac{3
   \sqrt{3}}{13}\right)\right)\right)\right)^{\frac{1}{2}}
\right\}$ \\ \hline
$8$ &$8x^7y-56x^5y^3+56x^3y^5-8xy^7$ &$\left\{ \pm 1, \pm\sqrt{3\pm2\sqrt2}\right\}$ \\ \hline
\multirow{2}{*}{$9$} & \multirow{2}{*}{$9x^8y-84x^6y^3+126x^4y^5-36x^2y^7+y^9$}&$\left\{ \pm\frac{\sqrt{3}}{3}, 
\pm\left(3+\frac{8 \cos \left(\frac{\pi }{18}\right)}{\sqrt{3}}\right)^{\frac{1}{2}}, \right.$\\
& & $\left.\pm \left(3\pm4 \sin \left(\frac{\pi }{18}\right)-\frac{4 \cos \left(\frac{\pi
   }{18}\right)}{\sqrt{3}}\right)^{\frac{1}{2}}\right\}$ \\ \hline
$10$ &$10x^9y-120x^7y^3+252x^5y^5-120x^3y^7+10xy^9$ &$\left\{ 
\pm \sqrt{1\pm\frac{2}{\sqrt5}} , 
\pm \sqrt{5\pm2\sqrt{5}} 
\right\}$ \\ \hline
\end{tabular}
\end{center}

\begin{example}
If $d=4$ and $\pi_{1}(\Omega_{I_{0}})=0$, a function $f=f_{0}+f_{v}\in\So(\Omega)$ 
which preserves no slice and has non-zero real part is
such that $f^{*4}$ is slice preserving if and only if $f_{v}$ does not have non-real isolated zeroes
of odd multiplicity and $f_{0}=\pm\sqrt{f_{v}^{s}}$.
\end{example}

\end{example}

\bibliographystyle{amsplain}

\end{document}